\documentclass[oneside, 11pt]{amsart} 

\usepackage{amsfonts}
\usepackage{amssymb}
\usepackage{latexsym}
\usepackage{graphics}
\usepackage[2emode]{psfrag}
\usepackage{amsthm}
\usepackage{amsmath}
\usepackage[all]{xy}
                
\def\*  C{{*  \mathcal C}}

\newcommand{\Ker}{{\rm Ker}\,}

\newcommand{\Img}{{\rm Im}\,}

\newcommand{\Cc}{\mathcal{C}}

\def\*  C{{}*  \hspace*  {-1pt}{\Cc}}

\def\text#1{{\rm {\rm #1}}}

\newtheorem{prop}{Proposition}[section] 
\newtheorem{lemma}[prop]{Lemma}
\newtheorem{cor}[prop]{Corollary}
\newtheorem{theo}[prop]{Theorem}

\theoremstyle{definition}
\newtheorem{Def}[prop]{Definition}
\newtheorem{ex}[prop]{Example}
\newtheorem{exs}[prop]{Examples}

\newtheorem{rems}[prop]{Remarks}
\newtheorem{rem}[prop]{Remark}

\newcommand{\Title}[1]{\bigskip\bigskip\centerline{\bf #1}\bigskip}
\newcommand{\Author}[1]{\medskip\centerline{ \it #1}}

\newcommand{\Affiliation}[1]{\medskip\centerline{#1}}
\newcommand{\Email}[1]{\medskip\centerline{#1}\bigskip}

\begin{document}  


\Title{COMMUTATIVE DEDUCTIVE SYSTEMS OF PSEUDO BCK-ALGEBRAS} 
\title[Commutative deductive systems of pseudo BCK-algebras]{}
                                                                           
\Author{\textbf{LAVINIA CORINA CIUNGU}}
\Affiliation{Department of Mathematics} 
\Affiliation{University of Iowa}
\Affiliation{14 MacLean Hall, Iowa City, Iowa 52242-1419, USA}
\Email{lavinia-ciungu@uiowa.edu}

\begin{abstract} 
In this paper we generalize the axiom systems given by M. Pa\l asi\'nski, B. Wo\'zniakowska and by W.H. Cornish 
for commutative BCK-algebras to the case of commutative pseudo BCK-algebras. 
A characterization of commutative pseudo BCK-algebras is also given.
We define the commutative deductive systems of pseudo BCK-algebras and we generalize some results proved 
by Yisheng Huang for commutative ideals of BCI-algebras to the case of commutative deductive systems of pseudo BCK-algebras. We prove that a pseudo BCK-algebra $A$ is commutative if and only if all the deductive systems of $A$ 
are commutative. We show that a normal deductive system $H$ of a pseudo BCK-algebra $A$ is commutative if and 
only if $A/H$ is a commutative pseudo BCK-algebra. We introduce the notions of state operators and state-morphism operators on pseudo BCK-algebras, and we apply these results on commutative deductive systems to investigate the properties of these operators. \\

\textbf{AMS classification (2000):} 03G25, 06F05, 06F35
\end{abstract}

\maketitle

\section{Introduction}

BCK-algebras were introduced by Y. Imai and K. Is\`eki in 1966 (\cite{Imai},\cite{Ise1}) as algebras with a 
binary operation $*$ modeling the set-theoretical difference and with a constant element $0$ that is a least element. 
Another motivation is from classical and non-classical propositional calculi modeling logical implications. 
Such algebras contain as a special subfamily the family of MV-algebras where important fuzzy structures can be studied. 
In 1975, S. Tanaka defined a special class of BCK-algebras called commutative BCK-algebras (\cite{Tan1}).  
H. Yutani was the first to give an equational base for commutative BCK-algebras (\cite{Yut1},\cite{Yut2}).  
An axiom system for commutative BCK-algebras consisting of three identities was given by M. Pa\l asi\'nski 
and B. Wo\'zniakowska in \cite{Pal1}, while W.H. Cornish gave in \cite{Cor1} an axiom system for commutative BCK-algebras consisting of two identities. 
For more details about BCK-algebras, see \cite{DvPu}, \cite{Meng1}. 
Commutative ideals in BCK-algebras were introduced in \cite{Meng2}, and they were generalized in \cite{Meng3} and 
\cite{Rez2} for the case of BCI-algebras and BE-algebras, respectively. 
This class of ideals proved to play an important role in the study of state BCK-algebras and state-morphism 
BCK-algebras (see \cite{Bor1}). 
Pseudo BCK-algebras were introduced by G. Georgescu and A. Iorgulescu in \cite{Geo15} as algebras 
with "two differences", a left- and right-difference, instead of one $*$ and with a constant element $0$ as the least element. Nowadays pseudo BCK-algebras are used in a dual form, with two implications, $\to$ and $\rightsquigarrow$ and with one constant element $1$, that is the greatest element. Thus such pseudo BCK-algebras are in the "negative cone" and are also called "left-ones". 
Commutative pseudo BCK-algebras were originally defined by G. Georgescu and A. Iorgulescu in \cite{Geo15} 
under the name of \emph{semilattice-ordered pseudo BCK-algebras} and properties of these structures were investigated 
by J. K{\"u}hr in \cite{Kuhr6}, \cite{Kuhr2}. 
An axiom system for bounded commutative pseudo BCK-algebras was presented by Walendziak in \cite{Wal1}.
Introduced by W. Dudek and Y.B. Jun (\cite{Dud1}), pseudo BCI-algebras are pseudo BCK-algebras having no greatest element.
According to \cite{Zhang1}, commutative pseudo BCI-algebras and commutative pseudo BCI-filters were defined 
and studied in \cite{Lu1}. 
Pseudo BE-algebras were introduced in \cite{Bor2} as generalizations of BE-algebras, and the commutative pseudo BE-algebras have recently been investigated in \cite{Ciu31}. 
In 1995, D. Mundici introduced an analogue of the probability measure on MV-algebras, called a \emph{state} (\cite{Mun1}), as the averaging process for formulas in \L ukasiewicz logic. After that, the states on other many-valued logic algebras have been intensively studied.  
Flaminio and Montagna were the first to present a unified approach to states and probabilistic many-valued logic in a 
logical and algebraic setting (\cite{Fla1}). They added a unary operation, called \emph{internal state} or 
\emph{state operator} to the language of MV-algebras which preserves the usual properties of states. 
A more powerful type of logic can be given by algebraic structures with internal states, and they are also very interesting varieties of universal algebras.
Di Nola and Dvure\v censkij introduced the notion of a state-morphism MV-algebra which is a stronger variation of 
a state MV-algebra (\cite{DiDv1}, \cite{DiDv1-e}).
The notion of a state operator was extended by Rach{\accent23u}nek and $\rm \check{S}alounov\acute{a}$ in \cite{Rac1} for the case of GMV-algebras (pseudo MV-algebras). State operators and state-morphism operators on BL-algebras were introduced and investigated in \cite{Ciu3} and subdirectly irreducible state-morphism BL-algebras were studied in \cite{Dvu5}. 
Dvure\v censkij, Rach{\accent23u}nek and $\rm \check{S}alounov\acute{a}$ introduced state R$\ell$-monoids and state-morphism R$\ell$-monoids (\cite{DvRa4},\cite{DvRa5}), while the state operators on pseudo BL-algebras and 
bounded pseudo-hoops have been investigated in \cite{Con1} and in \cite{Ciu4},\cite{Ciu30}, respectively.  
Recently, the state BCK-algebras and state-morphism BCK-algebras were defined and studied in \cite{Bor1}. \\
In this paper we generalize to the case of commutative pseudo BCK-algebras the axiom systems given by 
M. Pa\l asi\'nski and B. Wo\'zniakowska in \cite{Pal1} and by W.H. Cornish in \cite{Cor1} for commutative BCK-algebras. A characterization of commutative pseudo BCK-algebras is also given.
We define the commutative deductive systems of pseudo BCK-algebras and we investigate their properties by 
generalizing some results proved in \cite{Huang1} for the case of commutative ideals of BCI-algebras. 
We prove that a pseudo BCK-algebra $A$ is commutative if and only if all the deductive systems of $A$ are commutative. 
We show that a normal deductive system $H$ of a pseudo BCK-algebra $A$ is commutative if and only if $A/H$ is a commutative pseudo BCK-algebra. 
Inspired from \cite{Bor1}, we introduce the notions of state operators of type I and type II and state-morphism operators on pseudo BCK-algebras, and we apply the above mentioned results to investigate the properties of these operators. We show that the kernel of a type II state operator is a commutative deductive system. 
We define the notion of a normal state operator, proving that any type I or type II state operator on a commutative pseudo BCK-algebra is normal. 
It is proved that any normal type II state operator is a normal type I state operator, while a normal type I state 
operator is a normal type II state operator if its kernel is a commutative deductive system. 
In the last section we introduce and study the notion of a state-morphism operator on a pseudo BCK-algebra, showing 
that any normal type II state operator on a linearly ordered pseudo BCK-algebra is a state-morphism. 
For the case of a linearly ordered commutative pseudo BCK-algebra, we prove that any type I state operator 
is also a state-morphism operator.

$\vspace*{5mm}$

\section{Preliminaries on pseudo BCK-algebras}

A \emph{pseudo BCK-algebra} (more precisely, \emph{reversed left-pseudo BCK-algebra}) is a structure 
${\mathcal A}=(A,\leq,\rightarrow,\rightsquigarrow,1)$ where $\leq$ is a binary relation on $A$, $\rightarrow$ and $\rightsquigarrow$ are binary operations on $A$ and $1$ is an element of $A$ satisfying, for 
all $x, y, z \in A$, the  axioms:\\
$(psBCK_1)$ $x \rightarrow y \leq (y \rightarrow z) \rightsquigarrow (x \rightarrow z)$, $\:\:\:$ 
            $x \rightsquigarrow y \leq (y \rightsquigarrow z) \rightarrow (x \rightsquigarrow z);$ \\ 
$(psBCK_2)$ $x \leq (x \rightarrow y) \rightsquigarrow y$,$\:\:\:$ $x \leq (x \rightsquigarrow y) \rightarrow y;$ \\
$(psBCK_3)$ $x \leq x;$ \\
$(psBCK_4)$ $x \leq 1;$ \\
$(psBCK_5)$ if $x \leq y$ and $y \leq x$ then $x = y;$ \\
$(psBCK_6)$ $x \leq y$ iff $x \rightarrow y = 1$ iff $x \rightsquigarrow y = 1$. \\

Since the partial order $\leq$ is determined by any of the two ``arrows", we can eliminate ``$\leq$" from 
the signature and denote a pseudo BCK-algebra by ${\mathcal A}=(A,\rightarrow,\rightsquigarrow,1)$. \\
An equivalent definition of a pseudo BCK-algebra is given in \cite{Kuhr6}. \\
The structure ${\mathcal A}=(A,\rightarrow,\rightsquigarrow,1)$ of the type $(2,2,0)$ is a pseudo BCK-algebra iff it satisfies the following identities and quasi-identity, for all $x, y, z \in A$:\\ 
$(psBCK_1')$ $(x \rightarrow y) \rightsquigarrow [(y \rightarrow z) \rightsquigarrow (x \rightarrow z)]=1;$ \\
$(psBCK_2')$ $(x \rightsquigarrow y) \rightarrow [(y \rightsquigarrow z) \rightarrow (x \rightsquigarrow z)]=1;$ \\
$(psBCK_3')$ $1 \rightarrow x = x;$ \\
$(psBCK_4')$ $1 \rightsquigarrow x = x;$ \\
$(psBCK_5')$ $x \rightarrow 1 = 1;$ \\
$(psBCK_6')$ $(x\rightarrow y =1$ and $y \rightarrow x=1)$ implies $x=y$. \\
The partial order $\le$ is defined by $x \le y$ iff $x \rightarrow y =1$ (iff $x \rightsquigarrow y =1$). \\ 
If the poset $(A, \le)$ is a meet-semilattice then ${\mathcal A}$ is called a \emph{pseudo BCK-meet-semilattice} 
and we denote it by $\mathcal{A}=(A, \wedge, \rightarrow, \rightsquigarrow, 1)$.       
If $(A,\leq)$ is a lattice then we will say that ${\mathcal A}$ is a \emph{pseudo BCK-lattice} and 
it is denoted by $\mathcal{A}=(A, \wedge, \vee, \rightarrow, \rightsquigarrow, 1)$. \\

A pseudo BCK-algebra $\mathcal{A}=(A,\rightarrow,\rightsquigarrow,1)$ with a constant $a\in A$ (which 
can denote any element) is called a \emph{pointed pseudo BCK-algebra}.\\
A pointed pseudo BCK-algebra is denoted by $\mathcal{A}=(A,\rightarrow,\rightsquigarrow,a,1)$. 

A pseudo BCK-algebra $A$ is called \emph{bounded} if there exists an element $0\in A$ such that $0\le x$ 
for all $x\in A$. In a bounded pseudo BCK-algebra $(A, \rightarrow, \rightsquigarrow, 0, 1)$ we can define 
two negations: $x^{\rightarrow_0}=x\rightarrow 0$ and $x^{\rightsquigarrow_0}=x\rightsquigarrow 0$. 
A bounded pseudo BCK-algebra $A$ is called \emph{good} if it satisfies the identity 
$x^{\rightarrow_0\rightsquigarrow_0}=x^{\rightsquigarrow_0\rightarrow_0}$ for all $x\in A$. \\ 
An algebra $(A,\le,\rightarrow,\rightsquigarrow,1)$ satisfying $(psBCK_1)$, $(psBCK_2)$, $(psBCK_3)$, $(psBCK_5)$, 
$(psBCK_6)$ (or equivalently $(psBCK_1')$, $(psBCK_2')$, $(psBCK_3')$, $(psBCK_4')$, $(psBCK_6')$) is called a 
\emph{pseudo BCI-algebra} (for details, see \cite{Dud1}).  

\begin{lemma} \label{ps-eq-40} $\rm($\cite{Geo15}$\rm)$ In any pseudo BCK-algebra 
$(A,\rightarrow,\rightsquigarrow,1)$ the following hold for all $x, y, z\in A$: \\
$(1)$ $x\le y$ implies $z\rightarrow x \le z\rightarrow y$ and $z\rightsquigarrow x \le z\rightsquigarrow y;$ \\
$(2)$ $x\le y$ implies $y\rightarrow z \le x\rightarrow z$ and $y\rightsquigarrow z \le x\rightsquigarrow z;$ \\
$(3)$ $x\rightarrow y \le (z\rightarrow x)\rightarrow (z\rightarrow y)$ and 
      $x\rightsquigarrow y \le (z\rightsquigarrow x)\rightsquigarrow (z\rightsquigarrow y);$ \\
$(4)$ $x\rightarrow (y\rightsquigarrow z)=y\rightsquigarrow (x\rightarrow z)$ and 
      $x\rightsquigarrow (y\rightarrow z)=y\rightarrow (x\rightsquigarrow z);$ \\
$(5)$ $x\le y\rightarrow x$ and $x\le y\rightsquigarrow x;$ \\
$(6)$ $((x\rightarrow y)\rightsquigarrow y)\rightarrow y=x\rightarrow y$ and 
      $((x\rightsquigarrow y)\rightarrow y)\rightsquigarrow y=x\rightsquigarrow y$.  
\end{lemma}

For more details about the properties of a pseudo BCK-algebra we refer te reader to \cite{Ior14} and \cite{Ciu2}. \\
Let $A$ be a pseudo BCK-algebra. The subset $D \subseteq A$ is called a \emph{deductive system} of $A$ if 
it satisfies the following conditions:\\
$(i)$ $1 \in D;$ \\
$(ii)$ for all $x, y \in A$, if $x, x \rightarrow y \in D$ then $y \in D$. \\ 
Condition $(ii)$ is equivalent to the condition: \\
$(ii^{\prime})$ for all $x, y \in B$, if $x, x \rightsquigarrow y \in D$ then $y \in D$. \\ 
A deductive system $D$ of a pseudo BCK-algebra $A$ is said to be \emph{normal} if it satisfies the condition:\\
$(iii)$ for all $x, y \in A$, $x \rightarrow y \in D$ iff $x \rightsquigarrow y \in D$. \\
We will denote by ${\mathcal DS}(A)$ the set of all deductive systems and by ${\mathcal DS}_{n}(A)$ 
the set of all normal deductive systems of a pseudo BCK-algebra $A$. 
Obviously $\{1\}, A \in {\mathcal DS}(A), {\mathcal DS}_n(A)$ and 
${\mathcal DS}_n(A)\subseteq {\mathcal DS}(A)$. 
A pseudo BCK-algebra is called \emph{simple} if ${\mathcal DS}(A)=\{\{1\},A\}$. \\ 
Given $H\in {\mathcal DS}_n(A)$, the relation $\Theta_H$ on $A$ defined by $(x,y)\in \Theta_H$ iff 
$x\rightarrow y\in H$ and $y\rightarrow x\in H$ is a congruence on $A$. 
Then $H=[1]_{\Theta_H}$ and $A/H=(A/\Theta_H,\rightarrow, \rightsquigarrow,[1]_{\Theta_H})$ is a pseudo BCK-algebra  
and we write $x/H=[x]_{\Theta_H}$ for every $x\in A$ (see \cite{Kuhr6}). \\
The function $\pi_H: A \longrightarrow A/H$ defined by $\pi_H(x)=x/H$ for any $x\in A$ is a surjective homomorphism which is called the \emph{canonical projection} from $A$ to $A/H$. One can easily prove that $\Ker(\pi_H)=H$. \\
For every subset $X\subseteq A$, the smallest deductive system of $A$ containing $X$ (i.e. the intersection of 
all deductive systems $D\in{\mathcal DS}(A)$ such that $X\subseteq D$) is called the deductive 
system \emph{generated by} $X$, and it will be denoted by $[X)$. If $X=\{x\}$ we write $[x)$ instead of $[\{x\})$. \\
For all $x, y\in A$ and $n\in{\mathbb N}$ we define $x\rightarrow^n y$ and $x\rightsquigarrow^n y$ inductively 
as follows: \\
$\hspace*{2cm}$ $x\rightarrow^0 y=y$, $x\rightarrow^n y=x\rightarrow (x\rightarrow^{n-1}y)$ for $n\ge 1;$ \\
$\hspace*{2cm}$ $x\rightsquigarrow^0 y=y$, 
                $x\rightsquigarrow^n y=x\rightsquigarrow (x\rightsquigarrow^{n-1} y)$ for $n\ge 1$. \\
Pseudo BE-algebras were introduced in \cite{Bor2} as another generalization of pseudo BCK-algebras. 
A \emph{pseudo BE-algebra} is an algebra $(A, \rightarrow, \rightsquigarrow, 1)$ of the type $(2, 2, 0)$ 
such that the following axioms hold for all $x, y, z\in A$: \\
$(psBE_1)$ $x\rightarrow x=x\rightsquigarrow x=1,$ \\
$(psBE_2)$ $x\rightarrow 1=x\rightsquigarrow 1=1,$ \\
$(psBE_3)$ $1\rightarrow x=1\rightsquigarrow x=x,$ \\
$(psBE_4)$ $x\rightarrow (y\rightsquigarrow z)=y\rightsquigarrow (x\rightarrow z),$ \\
$(psBE_5)$ $x\rightarrow y=1$ iff $x\rightsquigarrow y=1$. \\ 
It was proved in \cite[Prop. 2.10]{Ciu31} that any pseudo BCK-algebra is a pseudo BE-algebra. \\
A \emph{pseudo-hoop} is an algebra $(A,\odot, \rightarrow, \rightsquigarrow,1)$ of the type $(2,2,2,0)$ such that for all $x, y, z \in A$: \\
$(psH_1)$ $x\odot 1=1\odot x=x;$\\
$(psH_2)$ $x\rightarrow x=x\rightsquigarrow x=1;$\\
$(psH_3)$ $(x\odot y) \rightarrow z=x\rightarrow (y\rightarrow z);$\\
$(psH_4)$ $(x\odot y) \rightsquigarrow z=y\rightsquigarrow (x\rightsquigarrow z);$\\
$(psH_5)$ $(x\rightarrow y)\odot x=(y\rightarrow x)\odot y=x\odot(x\rightsquigarrow y)=y\odot(y\rightsquigarrow x)$. \\
A pseudo-hoop is a meet-semilattice with \\
$\hspace*{2cm}$
$x\wedge y=(x\rightarrow y)\odot x=(y\rightarrow x)\odot y=x\odot(x\rightsquigarrow y)=y\odot(y\rightsquigarrow x)$. \\
A pseudo-hoop $(A,\odot, \rightarrow, \rightsquigarrow,1)$ is said to be a \emph{Wajsberg pseudo-hoop} if it satisfies the following conditions for all $x, y\in A$: \\
$(W_1)$ $(x\rightarrow y)\rightsquigarrow y=(y\rightarrow x)\rightsquigarrow x;$\\
$(W_2)$ $(x\rightsquigarrow y)\rightarrow y=(y\rightsquigarrow x)\rightarrow x$. \\
A pseudo-hoop $A$ is called \emph{basic} if it satisfies the following conditions for all $x, y, z\in A$: \\
$(B_1)$ $(x\rightarrow y)\rightarrow z \leq ((y\rightarrow x)\rightarrow z)\rightarrow z;$ \\
$(B_2)$ $(x\rightsquigarrow y)\rightsquigarrow z \leq ((y\rightsquigarrow x)\rightsquigarrow z)\rightsquigarrow z$. \\
Every pseudo-hoop is a pseudo BCK-meet-semilattice (\cite[Prop. 2.2]{Ciu2}), every Wajsberg pseudo-hoop is a basic 
pseudo-hoop (\cite[Prop. 4.9]{Geo16}), and every simple basic pseudo-hoop is a linearly ordered Wajsberg pseudo-hoop (\cite[Cor. 4.15]{Geo16}).

$\vspace*{5mm}$

\section{On commutative pseudo BCK-algebras}

Commutative pseudo BCK-algebras were originally defined by G. Georgescu and A. Iorgulescu in \cite{Geo15} 
under the name of \emph{semilattice-ordered pseudo BCK-algebras}, while properties of these structures were 
investigated by J. K{\"u}hr in \cite{Kuhr6}, \cite{Kuhr2}. 
In this section we present some equational bases for commutative pseudo BCK-algebras. 
We first recall the equational system proved by J. $\rm K \ddot{u}hr$ in \cite{Kuhr2} as a 
generalization of the equational system given by H. Yutani in \cite{Yut1}, \cite{Yut2} for commutative BCK-algebras. 
We generalize to the case of commutative pseudo BCK-algebras the axiom systems proved by M. Pa\l asi\'nski 
and B. Wo\'zniakowska in \cite{Pal1} and by W.H. Cornish in \cite{Cor1} for commutative BCK-algebras. 
A characterization of commutative pseudo BCK-algebras is also given. 

\begin{Def} \label{comm-ds-10} A pseudo BCK-algebra $(A,\rightarrow,\rightsquigarrow,1)$ is said to be 
\emph{commutative} if it satisfies the following conditions, for all $x, y\in A$: \\ 
$\hspace*{3cm}$ $(x\rightarrow y)\rightsquigarrow y=(y\rightarrow x)\rightsquigarrow x,$ \\
$\hspace*{3cm}$ $(x\rightsquigarrow y)\rightarrow y=(y\rightsquigarrow x)\rightarrow x$. 
\end{Def}

If $(A,\rightarrow,\rightsquigarrow,1)$ is a commutative pseudo BCK-algebra then $(A,\le)$ is a join-semilattice, where $x\vee y=(x\rightarrow y)\rightsquigarrow y=(x\rightsquigarrow y)\rightarrow y$ (see \cite{Geo15}). 
Conversely, if $(A,\le)$ is a join pseudo BCK-semilattice with 
$x\vee y=(x\rightarrow y)\rightsquigarrow y=(x\rightsquigarrow y)\rightarrow y$, from $x\vee y=y\vee x$, we have 
$(x\rightarrow y)\rightsquigarrow y=(y\rightarrow x)\rightsquigarrow x$ and 
$(x\rightsquigarrow y)\rightarrow y=(y\rightsquigarrow x)\rightarrow x$, for all $x, y\in A$, that is 
$A$ is commutative. 
We obtained the following result:  

\begin{prop} \label{comm-ds-20} A pseudo BCK-algebra $A$ is commutative if and only if $(A,\le)$ is a 
join semilattice with $x\vee y=(x\rightarrow y)\rightsquigarrow y=(x\rightsquigarrow y)\rightarrow y$. 
\end{prop}  

\begin{prop} \label{comm-ds-30} A pseudo BCK-algebra $(A,\rightarrow,\rightsquigarrow,1)$ is a commutative if and 
only if $(x\rightarrow y)\rightsquigarrow y\le (y\rightarrow x)\rightsquigarrow x$ and 
$(x\rightsquigarrow y)\rightarrow y\le (y\rightsquigarrow x)\rightarrow x$, for all $x, y\in A$.  
\end{prop}
\begin{proof} By interchanging $x$ and $y$ we obtain 
$(y\rightarrow x)\rightsquigarrow x\le (x\rightarrow y)\rightsquigarrow y$ and 
$(y\rightsquigarrow x)\rightarrow x\le (x\rightsquigarrow y)\rightarrow y$, for all $x, y\in A$, that is $A$ is 
commutative.
\end{proof}

\begin{rems} \label{comm-ds-30-05} 
$(1)$ Commutative pseudo BE-algebras have the same definition as pseudo BCK-algebras. 
Any commutative pseudo BE-algebra is a pseudo BCK-algebra (\cite[Th. 3.4]{Ciu31}). \\ 
$(2)$ A pseudo BCI-algebra $(A,\rightarrow,\rightsquigarrow,1)$ satisfying the conditions: \\  
$\hspace*{3cm}$ $(x\rightarrow y)\rightsquigarrow y=(y\rightarrow x)\rightsquigarrow x,$ \\
$\hspace*{3cm}$ $(x\rightsquigarrow y)\rightarrow y=(y\rightsquigarrow x)\rightarrow x$ \\
for all $x, y\in A$, is a pseudo BCK-algebra (\cite[Th. 3.6]{Dud1}). 
\end{rems}

We can see that in the definition of a commutative pseudo BCK-algebra, axiom $(psBCK_6')$ is a consequence of 
axioms $(psBCK_3')$, $(psBCK_4')$ and the commutativity. 
Indeed if $x\rightarrow y=1$ and $y\rightarrow x=1$ then we have 
$x=1\rightsquigarrow x=(y\rightarrow x)\rightsquigarrow x=(x\rightarrow y)\rightsquigarrow y=1\rightsquigarrow y=y$, 
that is $(psBCK_6')$. 
Hence the class of commutative pseudo BCK-algebras is a variety. 
It was proved in \cite[Th. 4.1.11]{Kuhr6} that the variety of commutative pseudo BCK-algebras is weakly regular, 
congruence and $3$-permutable. 
Since for commutative pseudo BCK-algebras axiom $(psBCK_6')$ is not independent, to find other axiom systems of commutative pseudo BCK-algebras is an interesting direction of research. 
We will generalize to the commutative pseudo BCK-algebras certain axiom systems proved for commutative BCK-algebras. 
The first equational base for commutative BCK-algebras was given in 1977 by H. Yutani, \cite{Yut1}, \cite{Yut2}. 
He proved that an algebra $(A,*,0)$ of type $(2,0)$ is a commutative BCK-algebra if and only if it satisfies the following identities for all $x, y, z\in A$: $(i)$ $x*x=0$, $(ii)$ $x*0=x$, $(iii)$ $(x*y)*z=(x*z)*y$, 
$(iv)$ $x*(x*y)=y*(y*x)$. 
This result was generalized by J. $\rm K \ddot{u}hr$ in \cite{Kuhr2} to the case of pseudo BCK-algebras. 

\begin{theo} \label{comm-ds-30-10} $\rm($\cite[Th. 4.2]{Kuhr2}$\rm)$ 
An algebra $(A,\rightarrow,\rightsquigarrow,1)$ of type $(2,2,0)$ is a commutative pseudo BCK-algebra if and only if  
it satisfies the following identities for all $x, y, z\in A$: \\
$(Y_1)$ $(x\rightarrow y)\rightsquigarrow y=(y\rightarrow x)\rightsquigarrow x$ and 
        $(x\rightsquigarrow y)\rightarrow y=(y\rightsquigarrow x)\rightarrow x;$ \\
$(Y_2)$ $x\rightarrow (y\rightsquigarrow z)=y\rightsquigarrow (x\rightarrow z);$ \\      
$(Y_3)$ $x\rightarrow x=x\rightsquigarrow x=1;$ \\
$(Y_4)$ $1\rightarrow x=1\rightsquigarrow x=x$. 
\end{theo}

It is easy to see that condition $(Y_3)$ is equivalent to $x\rightarrow 1=x\rightsquigarrow 1=1$, for all $x\in A$.  
Indeed if $x\rightarrow x=x\rightsquigarrow x=1$ then 
$1=x\rightsquigarrow x=(1\rightarrow x)\rightsquigarrow x=(x\rightarrow 1)\rightsquigarrow 1$, by $(Y_1)$. 
Then $x\rightarrow 1=x\rightarrow ((x\rightarrow 1)\rightsquigarrow 1)=
(x\rightarrow 1)\rightsquigarrow (x\rightarrow 1)=1$, by $(Y_2)$. Similarly $x\rightsquigarrow 1=1$. 
Conversely, if $x\rightarrow 1=x\rightsquigarrow 1=1$ then 
$x\rightarrow x=(1\rightsquigarrow x)\rightarrow x=(x\rightsquigarrow 1)\rightarrow 1=1\rightarrow 1=1$, 
by $(Y_1)$ and $(Y_4)$. Similarly $x\rightsquigarrow x=1$. 
We obtain the characterization of a commutative pseudo BCK-algebra given by 
J. $\rm K \ddot{u}hr$ in \cite{Kuhr6}.  

\begin{prop} \label{comm-ds-40} $\rm($\cite[Prop. 4.1.10]{Kuhr6}$\rm)$ 
An algebra $(A,\rightarrow,\rightsquigarrow,1)$ of type $(2,2,0)$ is a commutative pseudo BCK-algebra if and only if  
it satisfies the following identities for all $x, y, z\in A$: \\
$(K_1)$ $(x\rightarrow y)\rightsquigarrow y=(y\rightarrow x)\rightsquigarrow x$ and 
        $(x\rightsquigarrow y)\rightarrow y=(y\rightsquigarrow x)\rightarrow x;$ \\
$(K_2)$ $x\rightarrow (y\rightsquigarrow z)=y\rightsquigarrow (x\rightarrow z);$ \\      
$(K_3)$ $x\rightarrow 1=x\rightsquigarrow 1=1;$ \\
$(K_4)$ $1\rightarrow x=1\rightsquigarrow x=x$. 
\end{prop}

An axiom system for commutative BCK-algebras consisting of three identities was given by M. Pa\l asi\'nski 
and B. Wo\'zniakowska in \cite{Pal1}. They proved that an algebra $(A,*,0)$ of type $(2,0)$ is a commutative 
BCK-algebra if and only if it satisfies the following identities for all $x, y, z\in A$: 
$(i)$ $((x*y)*z)*((x*z)*y)=0$, $(ii)$ $x*(x*y)=y*(y*x)$, $(iii)$ $z*((x*y)*x)=z$ (see also \cite[Th. 2.1.7]{Huang1}). 
Following the idea used by M. Pa\l asi\'nski and B. Wo\'zniakowska, in what follows we give a generalization 
of this system to the case of commutative pseudo BCK-algebras. 

\begin{theo} \label{comm-ds-40-10} An algebra $(A,\rightarrow,\rightsquigarrow,1)$ of type $(2,2,0)$ is a commutative pseudo BCK-algebra if and only if it satisfies the following identities for all $x, y, z\in A$: \\
$(P_1)$ $(x\rightarrow (y\rightsquigarrow z))\rightarrow (y\rightsquigarrow (x\rightarrow z))=1$ and 
        $(x\rightsquigarrow (y\rightarrow z))\rightarrow (y\rightarrow (x\rightsquigarrow z))=1;$ \\
$(P_2)$ $(x\rightarrow y)\rightsquigarrow y=(y\rightarrow x)\rightsquigarrow x$ and  
        $(x\rightsquigarrow y)\rightarrow y=(y\rightsquigarrow x)\rightarrow x;$ \\   
$(P_3)$ $(x\rightarrow (y\rightsquigarrow x))\rightarrow z=
         (x\rightarrow (y\rightsquigarrow x))\rightsquigarrow z=z$ and \\
$\hspace*{0.7cm}$           
        $(x\rightsquigarrow (y\rightarrow x))\rightarrow z=
         (x\rightsquigarrow (y\rightarrow x))\rightsquigarrow z=z$. 
\end{theo} 
\begin{proof}
If $(A,\rightarrow,\rightsquigarrow,1)$ is a commutative pseudo BCK-algebra then one can easily show that 
conditions $(P_1)-(P_3)$ hold. \\
Conversely, consider an algebra $(A,\rightarrow,\rightsquigarrow,1)$ satisfying conditions $(P_1)-(P_3)$. 
We will prove that $A$ satisfies conditions $(Y_1)-(Y_4)$ from Theorem \ref{comm-ds-30-10}, so it is a 
commutative pseudo BCK-algebra. As mentioned above, we will follow the ideea used in \cite{Pal1}. \\
First we show that $x\rightarrow (y\rightsquigarrow x)=u\rightarrow (v\rightsquigarrow u)$ and 
$x\rightsquigarrow (y\rightarrow x)=u\rightsquigarrow (v\rightarrow u)$, for all $x, y, u, v\in A$. 
Indeed by $(P_3)$ and $(P_2)$, we have: \\
$\hspace*{1cm}$
$x\rightarrow (y\rightsquigarrow x)=(x\rightarrow (y\rightsquigarrow x))\rightsquigarrow 
                                    (x\rightarrow (y\rightsquigarrow x))$ \\
$\hspace*{3.2cm}$
$=((u\rightarrow (v\rightsquigarrow u))\rightarrow (x\rightarrow (y\rightsquigarrow x)))\rightsquigarrow 
                                                   (x\rightarrow (y\rightsquigarrow x))$ \\ 
$\hspace*{3.2cm}$
$=((x\rightarrow (y\rightsquigarrow x))\rightarrow (u\rightarrow (v\rightsquigarrow u)))\rightsquigarrow 
                                                   (u\rightarrow (v\rightsquigarrow u))$ \\ 
$\hspace*{3.2cm}$
$=(u\rightarrow (v\rightsquigarrow u))\rightsquigarrow (u\rightarrow (v\rightsquigarrow u))$ \\                       $\hspace*{3.2cm}$
$=u\rightarrow (v\rightsquigarrow u)$. \\
Similarly $x\rightsquigarrow (y\rightarrow x)=u\rightsquigarrow (v\rightarrow u)$. \\
Replacing $u$ with $x$ and $v$ with $u\rightarrow (v\rightsquigarrow u)$ in the identity 
$x\rightarrow (y\rightsquigarrow x)=u\rightarrow (v\rightsquigarrow u)$ and applying $(P_3)$ we get 
$x\rightarrow (y\rightsquigarrow x)=x\rightarrow ((u\rightarrow (v\rightsquigarrow u))\rightsquigarrow x)=
x\rightarrow x$. \\
Similarly $x\rightsquigarrow (y\rightarrow x)=x\rightsquigarrow x$, 
$u\rightarrow (v\rightsquigarrow u)=u\rightarrow u$, 
$u\rightsquigarrow (v\rightarrow u)=u\rightsquigarrow u$. 
It follows that $x\rightarrow x=u\rightarrow u$ and $x\rightsquigarrow x=u\rightsquigarrow u$. 
From the identity $x\rightarrow x=u\rightarrow u$, replacing $u$ by $x\rightarrow (x\rightsquigarrow x)$ and 
applying $(P_3)$ and $(P_1)$ we get: \\ 
$\hspace*{1cm}$
$x\rightarrow x=(x\rightarrow (x\rightsquigarrow x))\rightarrow (x\rightarrow (x\rightsquigarrow x))=
x\rightarrow (x\rightsquigarrow x)$ \\
$\hspace*{2.2cm}$
$=(x\rightsquigarrow (x\rightarrow x))\rightarrow (x\rightarrow (x\rightsquigarrow x))=1$. \\
Similarly from $x\rightsquigarrow x=u\rightsquigarrow u$, replacing $u$ by $x\rightsquigarrow (x\rightarrow x)$ 
we obtain: \\ 
$\hspace*{1cm}$
$x\rightsquigarrow x=(x\rightsquigarrow (x\rightarrow x))\rightsquigarrow (x\rightsquigarrow (x\rightarrow x))=
x\rightsquigarrow (x\rightarrow x)$ \\
$\hspace*{2.2cm}$
$=(x\rightarrow (x\rightsquigarrow x))\rightarrow (x\rightsquigarrow (x\rightarrow x))=1$, 
hence $(Y_3)$ is proved. \\
From $x\rightarrow x=1$, $x\rightarrow (y\rightsquigarrow x)=x\rightarrow x$ and $(P_3)$ we get 
$1\rightsquigarrow x=(x\rightarrow x)\rightsquigarrow x=(x\rightarrow (y\rightsquigarrow x))\rightsquigarrow x=x$. 
Similarly from $x\rightsquigarrow x=1$, $x\rightsquigarrow (y\rightarrow x)=x\rightsquigarrow x$ and $(P_3)$ we  
obtain $1\rightarrow x=1$, that is $(Y_4)$. 
If $x\rightarrow y=1$ and $y\rightarrow x=1$ then by $(Y_4)$ and $(P_2)$ we have 
$x=1\rightsquigarrow x=(y\rightarrow x)\rightsquigarrow x=(x\rightarrow y)\rightsquigarrow y=1\rightsquigarrow y=y$. 
It follows that $(psBCK_6')$ holds. 
Exchanging $x$ and $y$ in $(P_1)$ we have 
$(y\rightarrow (x\rightsquigarrow z))\rightarrow (x\rightsquigarrow (y\rightarrow z))=1$ and 
$(y\rightsquigarrow (x\rightarrow z))\rightarrow (x\rightarrow (y\rightsquigarrow z))=1$. 
Applying $(psBCK_6')$ it follows that  
$x\rightarrow (y\rightsquigarrow z)=y\rightsquigarrow (x\rightarrow z)$, that is $(Y_2)$. 
Since $(P_2)$ is in fact $(Y_1)$, we conclude that $A$ satisfies conditions $(Y_1)-(Y_4)$ from Theorem \ref{comm-ds-30-10}, thus it is a commutative pseudo BCK-algebra.  
\end{proof}

W.H. Cornish gave in \cite{Cor1} an axiom system for commutative BCK-algebras consisting of two identities, namely 
an algebra $(A,*,0)$ of type $(2,0)$ is a commutative BCK-algebra if and only if it satisfies the following 
identities for all $x, y, z\in A$: $(i)$ $x*(0*y)=x$, $(ii)$ $(x*z)*(x*y)=(y*z)*(y*x)$ 
(see also \cite[Th. 2.1.8]{Huang1}). 
In the next result we generalize Cornish's axiom system to the case of commutative pseudo BCK-algebras.  

\begin{theo} \label{comm-ds-50} An algebra $(A,\rightarrow,\rightsquigarrow,1)$ of type $(2,2,0)$ is a commutative pseudo BCK-algebra if and only if it satisfies the following identities for all $x, y, z\in A$: \\
$(C_1)$ $(x\rightarrow 1)\rightsquigarrow y=(x\rightsquigarrow 1)\rightarrow y=y;$ \\
$(C_2)$ $(x\rightarrow y)\rightsquigarrow (z\rightarrow y)=(y\rightarrow x)\rightsquigarrow (z\rightarrow x);$ \\
$(C_3)$ $(x\rightsquigarrow y)\rightarrow (z\rightsquigarrow y)=(y\rightsquigarrow x)\rightarrow (z\rightsquigarrow x)$.
\end{theo} 
\begin{proof}
Let $(A,\rightarrow,\rightsquigarrow,1)$ be a commutative pseudo BCK-algebra. 
By $(psBCK_5')$, $x\rightarrow 1=1$, so $x\le 1$, hence $x\rightsquigarrow 1=1$. 
Applying $(psBCK_3')$ we get $(x\rightarrow 1)\rightsquigarrow y=(x\rightsquigarrow 1)\rightarrow y=y$, 
that is $(C_1)$. 
From $(x\rightarrow y)\rightsquigarrow y=(y\rightarrow x)\rightsquigarrow x$ we have 
$z\rightarrow ((x\rightarrow y)\rightsquigarrow y)=z\rightarrow ((y\rightarrow x)\rightsquigarrow x)$, and 
applying Lemma \ref{ps-eq-40}$(4)$, it follows that 
$(x\rightarrow y)\rightsquigarrow (z\rightarrow y)=(y\rightarrow x)\rightsquigarrow (z\rightarrow x)$, 
that is $(C_2)$. 
Similarly from $(x\rightsquigarrow y)\rightarrow y=(y\rightsquigarrow x)\rightarrow x$ we get  
$(x\rightsquigarrow y)\rightarrow (z\rightsquigarrow y)=(y\rightsquigarrow x)\rightarrow (z\rightsquigarrow x)$, 
that is $(C_3)$. \\
Conversely, consider an algebra $(A,\rightarrow,\rightsquigarrow,1)$ satisfying 
conditions $(C_1)$, $(C_2)$, $(C_3)$. \\ 
Applying $(C_1)$ and $(C_2)$ we have: \\
$(y\rightarrow x)\rightsquigarrow x=(y\rightarrow x)\rightsquigarrow((y\rightsquigarrow 1)\rightarrow x)=
(x\rightarrow y)\rightsquigarrow((y\rightsquigarrow 1)\rightarrow y)=(x\rightarrow y)\rightsquigarrow y$. \\
Similarly $(y\rightsquigarrow x)\rightarrow x=(x\rightsquigarrow y)\rightarrow y$, hence the commutativity is proved. \\
Applying $(C_1)$ we get $(1\rightarrow 1)\rightsquigarrow 1=1$ and 
$1\rightarrow x=((1\rightarrow 1)\rightsquigarrow 1)\rightarrow x=x$, and similarly $1\rightsquigarrow x=x$. 
Hence $(psBCK_3')$ and $(psBCK_4')$ are proved. \\
By $(C_1)$ we have $((x\rightarrow 1)\rightarrow 1)\rightsquigarrow (1\rightarrow 1)=1\rightsquigarrow 1=1$. \\
Applying $(C_1)$ we get $(x\rightarrow 1)\rightsquigarrow (x\rightarrow 1)=x\rightarrow 1$ and 
$(x\rightsquigarrow 1)\rightarrow (x\rightsquigarrow 1)=x\rightsquigarrow 1$. \\
On the other hand, by $(psBCK_3')$ and $(C_2)$ we have: \\
$\hspace*{1cm}$
$(x\rightarrow 1)\rightsquigarrow (x\rightarrow 1)=
(1\rightarrow (x\rightarrow 1))\rightsquigarrow (1\rightarrow (x\rightarrow 1))$ \\
$\hspace*{4.4cm}$
$=((x\rightarrow 1)\rightarrow 1)\rightsquigarrow (1\rightarrow 1)=1\rightsquigarrow 1=1$. \\
Hence $x\rightarrow 1=1$ and similarly $x\rightsquigarrow 1=1$, that is $(psBCK_5')$. \\
By $(C_2)$ we have 
$(y\rightarrow z)\rightsquigarrow (x\rightarrow z)=(z\rightarrow y)\rightsquigarrow (x\rightarrow y)$, so \\
$\hspace*{0.5cm}$
$(x\rightarrow y)\rightarrow ((y\rightarrow z)\rightsquigarrow (x\rightarrow z))=
(x\rightarrow y)\rightarrow ((z\rightarrow y)\rightsquigarrow (x\rightarrow y))$ \\
$\hspace*{6cm}$
$=(1\rightsquigarrow (x\rightarrow y))\rightarrow ((z\rightarrow y)\rightsquigarrow (x\rightarrow y))$ \\
$\hspace*{6cm}$
$=((x\rightarrow y)\rightsquigarrow 1)\rightarrow ((z\rightarrow y)\rightsquigarrow 1)$ \\
$\hspace*{6cm}$
$=1\rightarrow 1=1$, \\
that is $(psBCK_1')$. Similarly $(psBCK_2')$. \\
Suppose that $x\rightarrow y=y\rightarrow x=1$. By $(psBCK_4')$ and by commutativity we have: \\
$y=1\rightsquigarrow y=(x\rightarrow y)\rightsquigarrow y=(y\rightarrow x)\rightsquigarrow x=
1\rightsquigarrow x=x$, that is $(psBCK_6')$. \\
We conclude that $(A,\rightarrow,\rightsquigarrow,1)$ is a commutative pseudo BCK-algebra.
\end{proof}

In the next result we give a characterization of commutative pseudo BCK-algebras. 

\begin{theo} \label{comm-ds-60} Let $(A,\rightarrow,\rightsquigarrow,1)$ be a pseudo BCK-algebra. 
The following are equivalent for all $x, y\in A$: \\
$(a)$ $A$ is commutative; \\
$(b)$ $x\rightarrow y=((y\rightarrow x)\rightsquigarrow x)\rightarrow y$ and 
      $x\rightsquigarrow y=((y\rightsquigarrow x)\rightarrow x)\rightsquigarrow y;$ \\
$(c)$ $(x\rightarrow y)\rightsquigarrow y=(((x\rightarrow y)\rightsquigarrow y)\rightarrow x)\rightsquigarrow x$ and 
      $(x\rightsquigarrow y)\rightarrow y=(((x\rightsquigarrow y)\rightarrow y)\rightsquigarrow x)\rightarrow x;$ \\
$(d)$ $x\le y$ implies $y=(y\rightarrow x)\rightsquigarrow x=(y\rightsquigarrow x)\rightarrow x$.
\end{theo}
\begin{proof}
$(a)$ $\Rightarrow (b)$ Since $A$ is commutative, applying Lemma \ref{ps-eq-40}$(6)$ we get:\\
$\hspace*{2cm}$ $x\rightarrow y=((x\rightarrow y)\rightsquigarrow y)\rightarrow y=
                                ((y\rightarrow x)\rightsquigarrow x)\rightarrow y$ and \\
$\hspace*{2cm}$ $x\rightsquigarrow y=((x\rightsquigarrow y)\rightarrow y)\rightsquigarrow y=
                                     ((y\rightsquigarrow x)\rightarrow x)\rightsquigarrow y$. \\                       
$(b)$ $\Rightarrow (c)$ By $(b)$ we have 
$(x\rightarrow y)\rightsquigarrow y=(((y\rightarrow x)\rightsquigarrow x)\rightarrow y)\rightsquigarrow y$. \\         
Exchanging $x$ and $y$ we get 
$(y\rightarrow x)\rightsquigarrow x=(((x\rightarrow y)\rightsquigarrow y)\rightarrow x)\rightsquigarrow x$. \\
By $(psBCK_2)$ we have 
$(x\rightarrow y)\rightsquigarrow y\le (((x\rightarrow y)\rightsquigarrow y)\rightarrow x)\rightsquigarrow x$. \\
It follows that $(x\rightarrow y)\rightsquigarrow y\le (y\rightarrow x)\rightsquigarrow x$. \\
Similarly $(y\rightarrow x)\rightsquigarrow x\le (x\rightarrow y)\rightsquigarrow y$, that is 
$(x\rightarrow y)\rightsquigarrow y=(y\rightarrow x)\rightsquigarrow x$. \\
Hence $(x\rightarrow y)\rightsquigarrow y=(((x\rightarrow y)\rightsquigarrow y)\rightarrow x)\rightsquigarrow x$. \\ 
Similarly $(x\rightsquigarrow y)\rightarrow y=(((x\rightsquigarrow y)\rightarrow y)\rightsquigarrow x)\rightarrow x$. \\
$(c)$ $\Rightarrow (d)$ Since $x\le y$ implies $x\rightarrow y=x\rightsquigarrow y=1$, applying $(c)$ we get 
$\hspace*{2cm}$ $y=(y\rightarrow x)\rightsquigarrow x$ and $y=(y\rightsquigarrow x)\rightarrow x$. \\
$(d)$ $\Rightarrow (a)$ From $x\le (x\rightarrow y)\rightsquigarrow y$, applying $(d)$ we have \\
$\hspace*{2cm}$ 
$(x\rightarrow y)\rightsquigarrow y=(((x\rightarrow y)\rightsquigarrow y)\rightarrow x)\rightsquigarrow x$. \\
From $y\le (x\rightarrow y)\rightsquigarrow y$, applying Lemma \ref{ps-eq-40}$(2)$ twice, we get: \\
$\hspace*{2cm}$ $((x\rightarrow y)\rightsquigarrow y)\rightarrow x \le y\rightarrow x$ and \\
$\hspace*{2cm}$ 
$(y\rightarrow x)\rightsquigarrow x\le(((x\rightarrow y)\rightsquigarrow y)\rightarrow x)\rightsquigarrow x$. \\ 
Hence $(y\rightarrow x)\rightsquigarrow x\le (x\rightarrow y)\rightsquigarrow y$. 
Similarly $(y\rightsquigarrow x)\rightarrow x\le (x\rightsquigarrow y)\rightarrow y$. \\ 
By Proposition \ref{comm-ds-30} it follows that $A$ is commutative. 
\end{proof}

Note that the equivalence $(a)$ $\Leftrightarrow (d)$ is also proved in \cite[Lemma 4.1.4]{Kuhr6}. 

\begin{ex} \label{comm-ds-60-10} (\cite[Ex. 4.1.2]{Kuhr2})
Let $(G, \vee,\wedge, \cdot, ^{-1}, e)$ be an $\ell$-group. On the negative cone $G^{-}=\{g\in G \mid g\le e\}$ we define the operations $x\rightarrow y=y\cdot (x\vee y)^{-1}$, $x\rightsquigarrow y=(x\vee y)^{-1}\cdot y$. 
Then $(G^{-}, \rightarrow, \rightsquigarrow, e)$ is a commutative pseudo BCK-algebra. 
\end{ex}

\begin{ex} \label{comm-ds-60-20} Let $(A,\rightarrow,\rightsquigarrow,1)$ be a pseudo BCK-algebra. 
A mapping $m:A\longrightarrow [0,\infty)$ such that $m(x\rightarrow y)=m(x\rightsquigarrow y)=m(y)-m(x)$ 
whenever $y\le x$ called a \emph{measure} on $A$. 
Then, by \cite[Prop. 4.3, Th. 4.8]{Ciu6}, $\Ker_0(m)=\{x\in A\mid m(x)=0\}\in {\mathcal DS}_{n}(A)$ and 
$A/\Ker_0(m)$ is a commutative pseudo BCK-algebra. 
\end{ex}

\begin{ex} \label{comm-ds-60-30} Let $(A,\rightarrow,\rightsquigarrow,1)$ be a pseudo BCK-algebra. 
A system ${\mathcal S}$ of measures on $A$ is an \emph{order-determing system} on $A$ if for all measures 
$m\in {\mathcal S}$, $m(x)\ge m(y)$ implies $x\le y$. If $A$ possesses an order-determing system ${\mathcal S}$ 
of measures then $A$ is commutative. \\
Indeed suppose that for $x, y\in A$ we have $x\le y$. Then, by \cite[Prop. 4.3]{Ciu6}, 
$m((y\rightsquigarrow x)\rightarrow x)=m((y\rightsquigarrow x)\rightarrow x)=m(y)$, for all $m\in {\mathcal S}$. 
Since ${\mathcal S}$ is order-determing then 
$(y\rightarrow x)\rightsquigarrow x=(y\rightsquigarrow x)\rightarrow x=y$. 
According to Theorem \ref{comm-ds-60}, $A$ is a commutative pseudo BCK-algebra. 
\end{ex}

\begin{ex} \label{comm-ds-60-35} Every Wajsberg pseudo-hoop is a commutative pseudo BCK-meet-semilattice.  
\end{ex}

\begin{rems} \label{comm-ds-60-40} 
$(1)$ Any finite commutative pseudo BCK-algebra is a BCK-algebra (see \cite[Cor. 3.6]{Kuhr2}). \\
$(2)$ If $(A,\rightarrow,\rightsquigarrow,0,1)$ is a bounded commutative pseudo BCK-algebra then 
$x^{\rightarrow_0\rightsquigarrow_0}=x^{\rightsquigarrow_0\rightarrow_0}=x$, that is $A$ is an \emph{involutive} 
pseudo BCK-algebra. 
More general, if $(A,\rightarrow,\rightsquigarrow,a,1)$ is a pointed commutative pseudo BCK-algebra then
$x^{-_a\sim_a}=x^{\sim_a-_a}=x$, for all $x\ge a$. 
\end{rems}

$\vspace*{5mm}$

\section{Commutative deductive systems of pseudo BCK-algebras}

In this section we define the commutative deductive systems of pseudo BCK-algebras and we investigate their 
properties by generalizing some results proved in \cite{Huang1} for the case of commutative ideals of BCI-algebras.
We prove that a pseudo BCK-algebra $A$ is commutative if and only if all deductive 
systems of $A$ are commutative. We show that a normal deductive system $H$ of a pseudo BCK-algebra $A$ is commutative if and only if $A/H$ is a commutative pseudo BCK-algebra. 

\begin{Def} \label{comm-ds-70} A deductive system $D$ of a pseudo BCK-algebra $(A,\rightarrow,\rightsquigarrow,1)$ 
is called \emph{commutative} if it satisfies the following conditions for all $x, y\in A$: \\ 
$(cds_1)$ $y\rightarrow x\in D$ implies $((x\rightarrow y)\rightsquigarrow y)\rightarrow x\in D;$ \\
$(cds_2)$ $y\rightsquigarrow x\in D$ implies $((x\rightsquigarrow y)\rightarrow y)\rightsquigarrow x\in D$. 
\end{Def}

We will denote by ${\mathcal DS}_c(A)$ the set of all commutative deductive systems of a pseudo BCK-algebra $A$. 

\begin{prop} \label{comm-ds-80} A subset $D$ of a pseudo BCK-algebra $(A,\rightarrow,\rightsquigarrow,1)$ is a 
commutative deductive system of $A$ if and only if it satisfies the following conditions for all $x, y, z\in A$: \\
$(1)$ $1\in D;$ \\
$(2)$ $z\rightarrow (y\rightarrow x)\in D$ and $z\in D$ implies 
      $((x\rightarrow y)\rightsquigarrow y)\rightarrow x\in D;$ \\
$(3)$ $z\rightsquigarrow (y\rightsquigarrow x)\in D$ and $z\in D$ implies 
      $((x\rightsquigarrow y)\rightarrow y)\rightsquigarrow x\in D$. 
\end{prop}
\begin{proof}
Let $D\in {\mathcal DS}_c(A)$. It follows that $1\in D$, that is condition $(1)$ is satisfied. \\
Consider $x, y, z\in A$ such that $z\rightarrow (y\rightarrow x)\in D$ and $z\in D$.  
Since $D$ is a deductive system, we have $y\rightarrow x\in D$, hence 
$((x\rightarrow y)\rightsquigarrow y)\rightarrow x\in D$, that is condition $(2)$.  
Similarly from $z\rightsquigarrow (y\rightsquigarrow x)\in D$ and $z\in D$ we get 
$((x\rightsquigarrow y)\rightarrow y)\rightsquigarrow x\in D$, that is condition $(3)$.  \\
Conversely, let $D$ be a subset of $A$ satisfying conditions $(1)$, $(2)$ and $(3)$. Obviously $1\in D$. \\ 
Consider $x, y\in D$ such that $x\rightarrow y\in D$ and $x\in D$. 
Since $x\rightarrow (1\rightarrow y)=x\rightarrow y\in D$, applying $(2)$ we get 
$y=((y\rightarrow 1)\rightsquigarrow 1)\rightarrow y\in D$. It follows that $D\in {\mathcal DS}(A)$. \\
Let $x, y\in A$ such that $y\rightarrow x\in D$. Since $1\rightarrow (y\rightarrow x)\in D$ and $1\in D$, 
by $(2)$ we get $((x\rightarrow y)\rightsquigarrow y)\rightarrow x\in D$.  
Similarly from $y\rightsquigarrow x\in D$ we obtain $((x\rightsquigarrow y)\rightarrow y)\rightsquigarrow x\in D$.  
We conclude that $D\in {\mathcal DS}_c(A)$. 
\end{proof}

\begin{ex} \label{comm-ds-90} Consider the set $A=\{0,a,b,c,d,1\}$ and the operations 
$\rightarrow,\rightsquigarrow$ given by the following tables:
\[
\hspace{10mm}
\begin{array}{c|ccccccc}
\rightarrow & 0 & a & b & c & d & 1 \\ \hline
0 & 1 & 1 & 1 & 1 & 1 & 1 \\ 
a & 0 & 1 & 1 & 1 & c & 1 \\ 
b & 0 & b & 1 & 1 & c & 1 \\ 
c & 0 & b & b & 1 & c & 1 \\ 
d & 0 & b & b & 1 & 1 & 1 \\
1 & 0 & a & b & c & d & 1 
\end{array}
\hspace{10mm} 
\begin{array}{c|ccccccc}
\rightsquigarrow & 0 & a & b & c & d & 1 \\ \hline
0 & 1 & 1 & 1 & 1 & 1 & 1 \\ 
a & 0 & 1 & 1 & 1 & c & 1 \\ 
b & 0 & c & 1 & 1 & c & 1 \\ 
c & 0 & a & b & 1 & c & 1 \\ 
d & 0 & a & b & 1 & 1 & 1 \\
1 & 0 & a & b & c & d & 1 
\end{array}
. 
\]
Then $(A,\rightarrow,\rightsquigarrow,1)$ is a pseudo BCK-algebra (see \cite[Ex. 3.1.4]{Kuhr6}). \\
One can see that 
${\mathcal DS}(A)=\{\{1\},\{1,c,d\},\{1,a,b,c,d\},A\}$, 
${\mathcal DS}_n(A)=\{\{1\},\{1,a,b,c,d\},A\}$,   
${\mathcal DS}_c(A)=\{\{1,a,b,c,d\},A\}$.
\end{ex}

\begin{ex} \label{comm-ds-100} Let $(A,\rightarrow,\rightsquigarrow,1)$ be a pseudo BCK-algebra and 
$m:A\longrightarrow [0,\infty)$ be a measure on $A$. 
It was proved in \cite[Prop. 4.2]{Ciu6} that 
$m((x\rightarrow y)\rightsquigarrow y)=m((y\rightarrow x)\rightsquigarrow x)$ and 
$m((x\rightsquigarrow y)\rightarrow y)=m((y\rightsquigarrow x)\rightarrow x)$, for all $x, y\in A$.
Consider $x, y\in A$ such that $y\rightarrow x\in \Ker_0(m)$, that is $m(y\rightarrow x)=0$. 
Since $x\le (x\rightarrow y)\rightsquigarrow y$ and $x\le y\rightarrow x$, we have 
$m(((x\rightarrow y)\rightsquigarrow y)\rightarrow x)=m(x)-m((x\rightarrow y)\rightsquigarrow y)=
                                                      m(x)-m((y\rightarrow x)\rightsquigarrow x)=
                                                      m(x)-m(x)+m(y\rightarrow x)=0$.   
Hence $((x\rightarrow y)\rightsquigarrow y)\rightarrow x\in \Ker_0(m)$. 
Similarly $y\rightsquigarrow x\in \Ker_0(m)$ implies 
$((x\rightsquigarrow y)\rightarrow y)\rightsquigarrow x\in \Ker_0(m)$. 
Thus $\Ker_0(m)\in {\mathcal DS}_{c}(A)$. 
\end{ex}

\begin{prop} \label{comm-ds-110} Let $(A,\rightarrow,\rightsquigarrow,1)$ be a pseudo BCK-algebra and 
$D\in {\mathcal DS}_c(A)$, $E\in {\mathcal DS}(A)$ such that $D\subseteq E$. Then $E\in {\mathcal DS}_c(A)$. 
\end{prop}
\begin{proof}
Consider $x, y\in A$ such that $u=y\rightarrow x\in E$. 
It follows that $y\rightarrow (u\rightsquigarrow x)=y\rightarrow ((y\rightarrow x)\rightsquigarrow x)=1\in D$. 
Since D is commutative, we have 
$(((u\rightsquigarrow x)\rightarrow y)\rightsquigarrow y)\rightarrow (u\rightsquigarrow x)\in D$. 
From $D\subseteq E$ we get 
$(((u\rightsquigarrow x)\rightarrow y)\rightsquigarrow y)\rightarrow (u\rightsquigarrow x)\in E$. 
Applying Lemma \ref{ps-eq-40}$(4)$, it follows that 
$u\rightsquigarrow ((((u\rightsquigarrow x)\rightarrow y)\rightsquigarrow y)\rightarrow x)\in E$. 
Since $u\in E$, we get $(((u\rightsquigarrow x)\rightarrow y)\rightsquigarrow y)\rightarrow x\in E$. 
From $x\le u\rightsquigarrow x$, we have $(u\rightsquigarrow x)\rightarrow y\le x\rightarrow y$, and 
$(x\rightarrow y)\rightsquigarrow y\le ((u\rightsquigarrow x)\rightarrow y)\rightsquigarrow y$, and finally 
$(((u\rightsquigarrow x)\rightarrow y)\rightsquigarrow y)\rightarrow x\le 
((x\rightarrow y)\rightsquigarrow y)\rightarrow x$. 
Hence $((x\rightarrow y)\rightsquigarrow y)\rightarrow x\in E$. 
Similarly from $y\rightsquigarrow x\in E$, we get $((x\rightsquigarrow y)\rightarrow y)\rightsquigarrow x\in E$. 
We conclude that $E\in {\mathcal DS}_c(A)$.
\end{proof}

\begin{cor} \label{comm-ds-120} The deductive system $\{1\}$ of a pseudo BCK-algebra $A$ is commutative if and 
only if ${\mathcal DS}(A)={\mathcal DS}_c(A)$.
\end{cor}

\begin{theo} \label{comm-ds-130} A pseudo BCK-algebra $A$ is commutative if and only if 
$\{1\}\in {\mathcal DS}_c(A)$.
\end{theo}
\begin{proof}
Assume that $A$ is commutative, thus $y\rightarrow x=((x\rightarrow y)\rightsquigarrow y)\rightarrow x$, 
for all $x, y \in A$. 
Consider $x, y\in A$ such that $y\rightarrow x\in \{1\}$. It follows that 
$((x\rightarrow y)\rightsquigarrow y)\rightarrow x\in \{1\}$. 
Similarly from $y\rightsquigarrow x\in \{1\}$ we get 
$((x\rightsquigarrow y)\rightarrow y)\rightsquigarrow x\in \{1\}$. 
Thus $\{1\}$ is a commutative deductive system of $A$. 
Conversely, assume that $\{1\}$ is a commutative deductive system of $A$ and let $x, y\in A$ such that $y\le x$, 
that is $y\rightarrow x=1\in \{1\}$. 
Since $\{1\}$ is commutative, it follows that $((x\rightarrow y)\rightsquigarrow y)\rightarrow x\in \{1\}$, 
that is $((x\rightarrow y)\rightsquigarrow y)\rightarrow x=1$. 
On the other hand $x\rightarrow ((x\rightarrow y)\rightsquigarrow y)=1$, 
hence $x=(x\rightarrow y)\rightsquigarrow y$. 
Applying Theorem \ref{comm-ds-60} it follows that $A$ is commutative.  
\end{proof}

\begin{cor} \label{comm-ds-140} A pseudo BCK-algebra $A$ is commutative if and only if 
${\mathcal DS}(A)={\mathcal DS}_c(A)$.
\end{cor}

\begin{theo} \label{comm-ds-150} Let $A$ be a pseudo BCK-algebra and $H\in {\mathcal DS}_n(A)$. 
Then $H\in {\mathcal DS}_c(A)$ if and only if $X/H$ is a commutative pseudo BCK-algebra.
\end{theo}
\begin{proof}
Assume $H\in {\mathcal DS}_c(A)$ and $x, y\in A$ such that $[y]_{\Theta_H}\rightarrow [x]_{\Theta_H}=[1]_{\Theta_H}$, 
so $[y\rightarrow x]_{\Theta_H}=[1]_{\Theta_H}$, that is $y\rightarrow x\in H$. 
Since $H$ is commutative, we get $((x\rightarrow y)\rightsquigarrow y)\rightarrow x\in H$, thus 
$(([x]_{\Theta_H}\rightarrow [y]_{\Theta_H})\rightsquigarrow [y]_{\Theta_H})\rightarrow [x]_{\Theta_H}= 
[((x\rightarrow y)\rightsquigarrow y)\rightarrow x]_{\Theta_H}=[1]_{\Theta_H}$. 
Hence $[1]_{\Theta_H}\in {\mathcal DS}_c(A/H)$, so $X/H$ is a commutative pseudo BCK-algebra. \\
Conversely, if $X/H$ is a commutative pseudo BCK-algebra then $[1]_{\Theta_H}\in {\mathcal DS}_c(A/H)$. 
If $y\rightarrow x\in H=[1]_{\Theta_H}$, we have $[y]_{\Theta_H}\rightarrow [x]_{\Theta_H}\in [1]_{\Theta_H}$. 
Since $[1]_{\Theta_H}$ is commutative, we get 
$(([x]_{\Theta_H}\rightarrow [y]_{\Theta_H})\rightsquigarrow [y]_{\Theta_H})\rightarrow [x]_{\Theta_H}=
[1]_{\Theta_H}$, so $[((x\rightarrow y)\rightsquigarrow y)\rightarrow x]_{\Theta_H}=[1]_{\Theta_H}$, that is 
$((x\rightarrow y)\rightsquigarrow y)\rightarrow x\in H$. 
Similarly from $y\rightsquigarrow x\in H$, we get 
$((x\rightsquigarrow y)\rightarrow y)\rightsquigarrow x\in H$. 
Hence $H\in {\mathcal DS}_c(A)$.
\end{proof}

$\vspace*{5mm}$

\section{State pseudo BCK-algebras}

Similarly as in \cite{Bor1} for the case of BCK-algebras, in this section we introduce the notions of state operators of type I and type II on pseudo BCK-algebras, and we apply the results proved in the previous sections to 
investigate the properties of state  operators. 
For the case of commutative pseudo BCK-algebras we show that the state operators of type I and II coincide. 
We show that the kernel of a type II state operator is a commutative deductive system. 
We define the notion of a normal state operator, proving that any type I or type II state operator on a 
commutative pseudo BCK-algebra is normal. 
We also prove that any normal type II state operator is a normal type I state operator, while a normal 
type I state operator is a normal type II state operator if its kernel is a commutative deductive system. 
 
\begin{Def} \label{st-psBCK-10} Let $(A, \rightarrow, \rightsquigarrow, 1)$ be a pseudo BCK-algebra   
and $\mu:A \longrightarrow A$ be a unary operator on $A$. For all $x, y\in A$ consider the following axioms:\\
$(IS_1)$ $\mu(x)\le \mu(y)$, whenever $x\le y,$ \\
$(IS_2)$ $\mu(x\rightarrow y)=\mu((x\rightarrow y)\rightsquigarrow y)\rightarrow \mu(y)$ and 
         $\mu(x\rightsquigarrow y)=\mu((x\rightsquigarrow y)\rightarrow y)\rightsquigarrow \mu(y),$ \\
$(IS^{'}_2)$ $\mu(x\rightarrow y)=\mu((y\rightarrow x)\rightsquigarrow x)\rightarrow \mu(y)$ and
             $\mu(x\rightsquigarrow y)=\mu((y\rightsquigarrow x)\rightarrow x)\rightsquigarrow \mu(y),$ \\ 
$(IS_3)$ $\mu(\mu(x)\rightarrow \mu(y))=\mu(x)\rightarrow \mu(y)$ and 
         $\mu(\mu(x)\rightsquigarrow \mu(y))=\mu(x)\rightsquigarrow \mu(y)$. \\
Then: \\
$(i)$ $\mu$ is called an \emph{internal state of type I} or a \emph{state operator of type I} or 
a \emph{type I state operator} if it satisfies axioms $(IS_1)$, $(IS_2)$, $(IS_3);$ \\    
$(ii)$ $\mu$ is called an \emph{internal state of type II} or a \emph{state operator of type II} or a 
\emph{type II state operator} if it satisfies axioms $(IS_1)$, $(IS^{'}_2)$, $(IS_3)$. \\
The structure $(A, \rightarrow, \rightsquigarrow, \mu, 1)$ ($(A,\mu)$, for short) is called a 
\emph{state pseudo BCK-algebra of type I (type II) state pseudo BCK-algebra}, respectively.
\end{Def}

Denote $\mathcal{IS}^{(I)}(A)$ and $\mathcal{IS}^{(II)}(A)$ the set of all internal states of 
type I and II on a pseudo BCK-algebra $A$, respectively. \\
For $\mu \in \mathcal{IS}^{(I)}(A)$ or $\mu \in \mathcal{IS}^{(II)}(A)$, 
$\Ker(\mu)=\{x\in A \mid \mu(x)=1\}$  is called the \emph{kernel} of $\mu$. \\
Note that the type I and II state operators are called in \cite{Bor1} left and right state operators, 
respectively. 

\begin{ex} \label{st-psBCK-10-20} Let $(A, \rightarrow, \rightsquigarrow, 1)$ be a pseudo BCK-algebra and 
$1_A, Id_A:A\longrightarrow A$, defined by $1_A(x)=1$ and $Id_A(x)=x$ for all $x\in A$. 
Then $1_A, Id_A\in \mathcal{IS}^{(I)}(A)$ and $1_A\in \mathcal{IS}^{(II)}(A)$. 
\end{ex}

\begin{ex} \label{st-psBCK-10-30} Let $(A_1, \rightarrow_1, \rightsquigarrow_1, 1_1)$ and 
$(A_2, \rightarrow_2, \rightsquigarrow_2, 1_2)$ be two pseudo BCK-algebras.
Denote $A=A_1 \times A_2=\{(x_1,x_2) \mid x_1\in A_1, x_2\in A_2\}$ and for all $(x_1, x_2), (y_1, y_2)\in A$, 
define the operations $\rightarrow, \rightsquigarrow, 1$ as follows: 
$(x_1, x_2)\rightarrow (y_1, y_2)=(x_1\rightarrow_1 y_1, x_2\rightarrow_2 y_2)$, 
$(x_1, x_2)\rightsquigarrow (y_1, y_2)=(x_1\rightsquigarrow_1 y_1, x_2\rightsquigarrow_2 y_2)$, 
$1=(1_1, 1_2)$. 
Obviously $(A, \rightarrow, \rightsquigarrow, 1)$ is a pseudo BCK-algebra. 
Consider $\mu_1\in \mathcal{IS}^{(I)}(A_1)$, $\mu_2\in \mathcal{IS}^{(I)}(A_2)$ and define the map $\mu:A\longrightarrow A$ by $\mu((x,y))=(\mu_1(x),\mu_2(x))$. Then $\mu\in \mathcal{IS}^{(I)}(A)$.  
Similarly if $\mu_1\in \mathcal{IS}^{(II)}(A_1)$, $\mu_2\in \mathcal{IS}^{(II)}(A_2)$, 
then $\mu\in \mathcal{IS}^{(II)}(A)$.
\end{ex}

\begin{prop} \label{st-psBCK-10-40}
A pseudo BCK-algebra $A$ is commutative if and only if $\mathcal{IS}^{(I)}(A)=\mathcal{IS}^{(II)}(A)$. 
\end{prop}
\begin{proof}
It is clear that if $A$ is commutative then $\mathcal{IS}^{(I)}(A)=\mathcal{IS}^{(II)}(A)$. 
Conversely, suppose that $\mathcal{IS}^{(I)}(A)=\mathcal{IS}^{(II)}(A)$. 
Since $Id_A\in \mathcal{IS}^{(I)}(A)$, we have $Id_A\in \mathcal{IS}^{(II)}(A)$, so 
$x\rightarrow y=((y\rightarrow x)\rightsquigarrow x)\rightarrow y$ and 
$x\rightsquigarrow y=((y\rightsquigarrow x)\rightarrow x)\rightsquigarrow y$, for all $x, y\in A$.  
According to Theorem \ref{comm-ds-60}, it follows that $A$ is a commutative pseudo BCK-algebra. 
\end{proof}

\begin{prop} \label{st-psBCK-20} Let $(A, \rightarrow, \rightsquigarrow, \mu, 1)$ be a type I or a type II state 
pseudo BCK-algebra. Then the following hold:\\ 
$(1)$ $\mu(1)=1;$ \\
$(2)$ $\mu(\mu(x))=\mu(x)$, for all $x\in A;$ \\
$(3)$ $\mu(x\rightarrow y)\le \mu(x)\rightarrow \mu(y)$ and 
      $\mu(x\rightsquigarrow y)\le \mu(x)\rightsquigarrow \mu(y)$, for all $x, y\in A;$ \\
$(4)$ $\Ker(\mu)\in {\mathcal DS}(A);$ \\
$(5)$ $\Img(\mu)$ is a subalgebra of $A;$ \\ 
$(6)$ $\Img(\mu)=\{x\in A \mid x=\mu(x)\};$ \\
$(7)$ $\Ker(\mu)\cap \Img(\mu)=\{1\};$ \\
$(8)$ if $A$ is commutative then 
         $\mu(x\rightarrow y)\rightsquigarrow \mu(y)=\mu(x\rightsquigarrow y)\rightarrow \mu(y)$, 
         for all $x, y\in A$. 
\end{prop}
\begin{proof}
$(1)$ From $(IS_2)$ and $(IS^{'}_2)$ for $x=y=1$ we get $\mu(1)=\mu(1)\rightarrow \mu(1)=1$. \\
$(2)$ Applying $(1)$ and $(IS_3)$ we have: \\
$\hspace*{1cm}$
$\mu(\mu(x))=\mu(1\rightarrow \mu(x))=\mu(\mu(1)\rightarrow \mu(x))=
\mu(1)\rightarrow \mu(x)=1\rightarrow \mu(x)=\mu(x)$. \\
$(3)$ If $\mu\in \mathcal{IS}^{(I)}(A)$ then from $x\le (x\rightarrow y)\rightsquigarrow y$ we 
get $\mu(x)\le \mu((x\rightarrow y)\rightsquigarrow y)$, so 
$\mu((x\rightarrow y)\rightsquigarrow y)\rightarrow \mu(y)\le \mu(x)\rightarrow \mu(y)$, that is 
$\mu(x\rightarrow y)\le \mu(x)\rightarrow \mu(y)$. \\
Similarly $\mu(x\rightsquigarrow y)\le \mu(x)\rightsquigarrow \mu(y)$. \\
If $\mu\in \mathcal{IS}^{(II)}(A)$ then from $x\le (y\rightarrow x)\rightsquigarrow x$ we have 
$\mu(x)\le \mu((y\rightarrow x)\rightsquigarrow x)$ and in a similar way we get 
$\mu(x\rightarrow y)\le \mu(x)\rightarrow \mu(y)$ and 
$\mu(x\rightsquigarrow y)\le \mu(x)\rightsquigarrow \mu(y)$. \\
$(4)$ Consider $x, x\rightarrow y\in \Ker(\mu)$, that is $\mu(x)=\mu(x\rightarrow y)=1$. \\
Applying $(3)$ we have $1=\mu(x\rightarrow y)\le \mu(x)\rightarrow \mu(y)$, so $\mu(x)\rightarrow \mu(y)=1$.  \\
It follows that $\mu(y)=1\rightarrow \mu(y)=\mu(1)\rightarrow \mu(y)=\mu(x)\rightarrow \mu(y)=1$, 
hence  $y\in \Ker(\mu)$. \\
Since $1\in \Ker(\mu)$, it follows that $\Ker(\mu)\in {\mathcal DS}(A)$. \\
$(5)$ Since $1=\mu(1)\in \Img(\mu)$, we have $1\in \Img(\mu)$. \\
If $x, y\in \Img(\mu)$ then from $\mu(x)\rightarrow \mu(y)=\mu(\mu(x)\rightarrow \mu(y))$ and  
$\mu(x)\rightsquigarrow \mu(y)=\mu(\mu(x)\rightsquigarrow \mu(y))$, it follows that 
$\mu(x)\rightarrow \mu(y), \mu(x)\rightsquigarrow \mu(y)\in \Img(\mu)$. 
Thus $\Img(\mu)$ is a subalgebra of $A$. \\
$(6)$ Clearly $\{x\in A \mid x=\mu(x)\} \subseteq \Img(\mu)$. 
Let $x\in \Img(\mu)$, that is there exists $x_1\in A$ such that $x=\mu(x_1)$. 
It follows that $x=\mu(x_1)=\mu(\mu(x_1))=\mu(x)$, that is $x\in \Img(\mu)$. \\
Thus $\Img(\mu) \subseteq \{x\in A \mid x=\mu(x)\}$ and we  conclude that $\Img(\mu)=\{x\in A \mid x=\mu(x)\}$. \\
$(7)$ Let $y\in \Ker(\mu)\cap \Img(\mu)$, so $\mu(y)=1$ and there exists $x\in A$ such that $\mu(x)=y$. \\
It follows that $y=\mu(x)=\mu(\mu(x))=\mu(y)=1$, thus $\Ker(\mu)\cap \Img(\mu)=\{1\}$. \\ 
$(8)$ If $A$ is commutative then it is is a join-semilattice, where 
$x\vee y=(x\rightarrow y)\rightsquigarrow y=(x\rightsquigarrow y)\rightarrow y$. 
Applying Lemma \ref{ps-eq-40}$(6)$ we have: \\
$\hspace*{1cm}$
$\mu(x\vee y)=\mu((x\rightarrow y)\rightsquigarrow y)=
              \mu(((x\rightarrow y)\rightsquigarrow y)\rightarrow y)\rightsquigarrow \mu(y)=
              \mu(x\rightarrow y)\rightsquigarrow \mu(y)$ and \\
$\hspace*{1cm}$
$\mu(x\vee y)=\mu((x\rightsquigarrow y)\rightarrow y)=
              \mu(((x\rightsquigarrow y)\rightarrow y)\rightsquigarrow y)\rightarrow \mu(y)=
              \mu(x\rightsquigarrow y)\rightarrow \mu(y)$. \\              
Hence $\mu(x\rightarrow y)\rightsquigarrow \mu(y)=\mu(x\rightsquigarrow y)\rightarrow \mu(y)$.               
\end{proof}

\begin{prop} \label{st-psBCK-20-10} Let $(A, \rightarrow, \rightsquigarrow, \mu, 1)$ be a commutative 
type I state pseudo BCK-algebra. Then $\Ker(\mu)\in {\mathcal DS}_c(A)$.  
\end{prop}
\begin{proof} It is a consequence of Proposition \ref{st-psBCK-20}$(4)$ and Corollary \ref{comm-ds-140}.
\end{proof}

\begin{prop} \label{st-psBCK-30} Let $(A, \rightarrow, \rightsquigarrow, \mu, 1)$ be a type II pseudo BCK-algebra. 
Then the following hold:\\ 
$(1)$ $y\le x$ implies $\mu(x\rightarrow y)=\mu(x)\rightarrow \mu(y)$ and 
                       $\mu(x\rightsquigarrow y)=\mu(x)\rightsquigarrow \mu(y);$ \\
$(2)$ $x\rightarrow y\in \Ker(\mu)$ iff $((y\rightarrow x)\rightsquigarrow x)\rightarrow y\in \Ker(\mu);$ \\            $(3)$ $x\rightsquigarrow y\in \Ker(\mu)$ iff $((y\rightsquigarrow x)\rightarrow x)\rightsquigarrow y\in \Ker(\mu);$ \\
$(4)$ $\Ker(\mu)\in {\mathcal DS}_c(A)$.      
\end{prop}
\begin{proof}
$(1)$ If $y\le x$, we have:\\
$\hspace*{1cm}$
$\mu(x\rightarrow y)=\mu((y\rightarrow x)\rightsquigarrow x)\rightarrow \mu(y)=
\mu(1\rightsquigarrow x)\rightarrow \mu(y)=\mu(x)\rightarrow \mu(y)$. \\  
Similarly $\mu(x\rightsquigarrow y)=\mu(x)\rightsquigarrow \mu(y)$. \\
$(2)$ Suppose $x\rightarrow y\in \Ker(\mu)$, that is $\mu(x\rightarrow y)=1$. 
Since $y\le (y\rightarrow x)\rightsquigarrow x$, applying $(1)$ we get   
$\mu(((y\rightarrow x)\rightsquigarrow x)\rightarrow y)=\mu((y\rightarrow x)\rightsquigarrow x)\rightarrow \mu(y)$, 
hence \\
$1=\mu(x\rightarrow y)=\mu((y\rightarrow x)\rightsquigarrow x)\rightarrow \mu(y)=
\mu(((y\rightarrow x)\rightsquigarrow x)\rightarrow y)$. \\ 
Thus $((y\rightarrow x)\rightsquigarrow x)\rightarrow y\in \Ker(\mu)$. \\ 
Conversely, suppose $((y\rightarrow x)\rightsquigarrow x)\rightarrow y\in \Ker(\mu)$. \\
From $x\le (y\rightarrow x)\rightsquigarrow x$ we have
$((y\rightarrow x)\rightsquigarrow x) \rightarrow y\le x\rightarrow y$. \\
Since $\Ker(\mu)$ is a deductive system of $A$, we get $x\rightarrow y\in \Ker(\mu)$. \\
$(3)$ Similarly as $(2)$. \\
$(4)$ It follows from $(2)$ and $(3)$. 
\end{proof}

\begin{Def} \label{st-psBCK-30-10} An internal state $\mu$ of type I or II on a pseudo BCK-algebra $A$ is said 
to be \emph{normal} if $\Ker(\mu)$ is a normal deductive system of $A$. In this case $(A,\mu)$ is said to be a 
\emph{normal type I(type II) state pseudo BCK-algebra}. 
\end{Def}

Denote $\mathcal{IS}^{(I)}_n(A)$ and $\mathcal{IS}^{(II)}_n(A)$ the set of all normal internal states of 
type I and II on a pseudo BCK-algebra $A$, respectively. 

\begin{prop} \label{st-psBCK-30-20} Any internal state of type I or type II on a commutative pseudo BCK-algebra 
is normal.  
\end{prop}
\begin{proof}
Let $\mu$ be an internal state of type I of a commutative pseudo BCK-algebra $A$. 
Consider $x, y\in A$ such that $x\rightarrow y\in \Ker(\mu)$, that is $\mu(x\rightarrow y)=1$. 
It follows that $\mu((x\rightarrow y)\rightsquigarrow y)\rightarrow \mu(y)=1$. 
Since $A$ is commutative, we get 
$\mu((x\rightsquigarrow y)\rightarrow y)\rightarrow \mu(y)=1$, that is $\mu(x\rightsquigarrow y)=1$, hence
$x\rightsquigarrow y\in \Ker(\mu)$. 
Similarly from $x\rightsquigarrow y\in \Ker(\mu)$ we get $x\rightarrow y\in \Ker(\mu)$. 
Hence $\Ker(\mu)$ is a normal deductive system on $A$. 
The case of an internal state of type II on $A$ can be treated in the same way as the case of type I internal state. 
\end{proof}

\begin{prop} \label{st-psBCK-30-30} Let $(A, \mu)$ be a a type I or type II state pseudo BCK-algebra.  
If $A$ is a meet-semilattice and $\mu$  satisfies the condition: \\
$\hspace*{2cm}$ $\mu(x\rightarrow y)=\mu(x)\rightarrow \mu(x\wedge y)$ and  
                     $\mu(x\rightsquigarrow y)=\mu(x)\rightsquigarrow \mu(x\wedge y)$, \\
for all $x, y\in A$ then $\mu$ is a normal state operator on $A$.     
\end{prop} 
\begin{proof}
Consider $x, y\in A$ such that $x\rightarrow y\in \Ker(\mu)$, that is $\mu(x\rightarrow y)=1$. 
It follows that $\mu(x)\rightarrow \mu(x\wedge y)=1$, so $\mu(x) \le \mu(x\wedge y)$, hence $\mu(x)=\mu(x\wedge y)$. 
We get $\mu(x\rightsquigarrow y)=\mu(x)\rightsquigarrow \mu(x\wedge y)=1$, that is $x\rightsquigarrow y\in \Ker(\mu)$. 
Similaly, from $x\rightsquigarrow y\in \Ker(\mu)$ we have $x\rightarrow y\in \Ker(\mu)$. 
We conclude that $\Ker(\mu)$ is a normal deductive system of $A$, thus $\mu$ is a normal state operator on $A$.
\end{proof} 

\begin{Def} \label{st-psBCK-30-40} Let $(A,\wedge,\rightarrow,\rightsquigarrow,1)$ be a pseudo BCK-meet-semilattice 
and $\mu$ be a unary operator on $A$. Then $(A,\mu)$ is called a \emph{state pseudo BCK-meet-semilattice of type I (type II)} or a \emph{type I (type II) state pseudo BCK-meet-semilattice} if $(A,\mu)$ is a type I (type II) 
state pseudo BCK-algebra and $\mu$ satisfies the following condition, for all $x, y\in A$: \\ 
$(IS_4)$ $\mu(\mu(x)\wedge \mu(y))=\mu(x)\wedge \mu(y)$. 
\end{Def}

\begin{prop} \label{st-psBCK-30-50} Let $(A,\wedge,\rightarrow,\rightsquigarrow,1)$ be a linearly ordered pseudo BCK-meet-semilattice and $\mu$ be a unary operator on $A$. Then: \\
$(1)$ if $(A,\mu)$ is a type I (type II) state pseudo BCK-algebra then $(A,\mu)$ is a type I (type II) state pseudo BCK-meet-semilattice; \\
$(2)$ if $(A,\mu)$ is a normal type I (type II) state pseudo BCK-algebra then $(A,\mu)$ is a normal type I (type II) state pseudo BCK-meet-semilattice.  
\end{prop} 
\begin{proof}
$(1)$ It is straightforward. \\
$(2)$ It is a consequence of Proposition \ref{st-psBCK-30-30}.
\end{proof}

\begin{cor} \label{st-psBCK-30-60} Let $(A,\odot,\rightarrow,\rightsquigarrow,1)$ be a simple basic pseudo-hoop. 
Then: \\
$(1)$ if $(A,\mu)$ is a type I (type II) state pseudo BCK-algebra then $(A,\mu)$ is a type I (type II) state pseudo BCK-meet-semilattice; \\
$(2)$ if $(A,\mu)$ is a normal type I (type II) state pseudo BCK-algebra then $(A,\mu)$ is a normal type I (type II) state pseudo BCK-meet-semilattice.  
\end{cor}

\begin{prop} \label{st-psBCK-40} Let $(A, \rightarrow, \rightsquigarrow, \mu, 1)$ be a normal type II pseudo BCK-algebra. Then: \\
$(1)$ the map $\hat \mu:A/\Ker(\mu) \to A/\Ker(\mu)$ defined by $\hat \mu(x/\Ker(\mu))= \mu(x)/\Ker(\mu)$ is both 
a normal type I and type II state operator on $A/\Ker(\mu);$ \\
$(2)$ $(A, \mu)$ is a normal type I state pseudo BCK-algebra.
\end{prop} 
\begin{proof}
$(1)$ If $x/\Ker(\mu)=y/\Ker(\mu)$ then $x\rightarrow y, y\rightarrow x\in \Ker(\mu)$, that is 
$\mu(x\rightarrow y)=\mu(y\rightarrow x)=1$. Applying Proposition \ref{st-psBCK-20}$(3)$, it follows that 
$\mu(x)=\mu(y)$. Hence $\hat \mu$ is well defined. 
The proof of the fact that $\hat \mu$ is a normal type II state on $A/\Ker(\mu)$ is straightforward. 
Since $\Ker(\mu)$ is a normal commutative deductive system of $A$ then according to Theorem \ref{comm-ds-150}, $A/\Ker(\mu)$ is a commutative pseudo BCK-algebra. By Proposition \ref{st-psBCK-10-40}, $\hat \mu$ is also a 
normal type I state on $A/\Ker(\mu)$. \\
$(2)$ Let $x, y\in A$. Since $A/\Ker(\mu)$ is a commutative pseudo BCK-algebra, it follows that \\
$\hspace*{2cm}$ $((y\rightarrow x)\rightsquigarrow x)/\Ker(\mu)=((x\rightarrow y)\rightsquigarrow y)/\Ker(\mu)$ and \\
$\hspace*{2cm}$ $((y\rightsquigarrow x)\rightarrow x)/\Ker(\mu)=((x\rightsquigarrow y)\rightarrow y)/\Ker(\mu)$. \\
Similarly as in $(1)$ we get \\
$\hspace*{2cm}$ $\mu((y\rightarrow x)\rightsquigarrow x)=\mu((x\rightarrow y)\rightsquigarrow y)$ and \\
$\hspace*{2cm}$ $\mu((y\rightsquigarrow x)\rightarrow x)=\mu((x\rightsquigarrow y)\rightarrow y)$. \\
Hence \\
$\hspace*{2cm}$ $\mu((y\rightarrow x)\rightsquigarrow x)\rightarrow \mu(x)=
                 \mu((x\rightarrow y)\rightsquigarrow y)\rightarrow \mu(y)$ and \\
$\hspace*{2cm}$ $\mu((y\rightsquigarrow x)\rightarrow x)\rightsquigarrow \mu(x)=
                  \mu((x\rightsquigarrow y)\rightarrow y)\rightsquigarrow \mu(y)$. \\
Thus $(A, \mu)$ is a normal type I state pseudo BCK-algebra.
\end{proof}

\begin{theo} \label{st-psBCK-50} Let $\mu:A\longrightarrow A$ be a map on a pseudo BCK-algebra $A$. 
Then $(A, \mu)$ is a normal type II state pseudo BCK-algebra if and only if $(A, \mu)$ is a normal type I 
state pseudo BCK-algebra and $\Ker(\mu)\in {\mathcal DS}_c(A)$.
\end{theo}
\begin{proof}
Suppose that $(A, \mu)$ is a normal type II state pseudo BCK-algebra. Then, according to Proposition \ref{st-psBCK-40}, 
$(A, \mu)$ is a normal type I state pseudo BCK-algebra. 
Conversely, let $(A, \mu)$ be a normal type I state pseudo BCK-algebra such that $\Ker(\mu)$ is a commutative deductive system of $A$. It follows that \\
$\hspace*{2cm}$ $((y\rightarrow x)\rightsquigarrow x)/\Ker(\mu)=((x\rightarrow y)\rightsquigarrow y)/\Ker(\mu)$ and \\
$\hspace*{2cm}$ $((y\rightsquigarrow x)\rightarrow x)/\Ker(\mu)=((x\rightsquigarrow y)\rightarrow y)/\Ker(\mu)$, \\
for all $x, y\in A$. Similarly as in the proof of Proposition \ref{st-psBCK-40}, we have \\
$\hspace*{2cm}$ $\mu((y\rightarrow x)\rightsquigarrow x)\rightarrow \mu(x)=
                 \mu((x\rightarrow y)\rightsquigarrow y)\rightarrow \mu(y)$ and \\
$\hspace*{2cm}$ $\mu((y\rightsquigarrow x)\rightarrow x)\rightsquigarrow \mu(x)=
                  \mu((x\rightsquigarrow y)\rightarrow y)\rightsquigarrow \mu(y)$. \\
Thus $(A, \mu)$ is a normal type II state pseudo BCK-algebra.
\end{proof}

\begin{ex} \label{st-psBCK-70} Consider the pseudo BCK-algebra $(A,\rightarrow,\rightsquigarrow,1)$ 
from Example \ref{comm-ds-90} and the maps $\mu_i:A\longrightarrow A$, $i=1,2,\cdots,10$, given in the table below:
\[
\begin{array}{c|cccccc}
 x & 0 & a & b & c & d & 1 \\ \hline
\mu_1(x) & 0 & 0 & 0 & 1 & 1 & 1 \\
\mu_2(x) & 0 & a & a & 1 & 1 & 1 \\
\mu_3(x) & 0 & a & b & c & d & 1 \\
\mu_4(x) & 0 & b & b & 1 & 1 & 1 \\
\mu_5(x) & 0 & d & c & c & d & 1 \\
\mu_6(x) & 0 & 1 & 1 & 1 & 1 & 1 \\
\mu_7(x) & a & a & a & 1 & 1 & 1 \\
\mu_8(x) & b & b & b & 1 & 1 & 1 \\
\mu_9(x) & d & d & c & c & d & 1 \\
\mu_{10}(x) & 1 & 1 & 1 & 1 & 1 & 1 
\end{array}
.   
\]
Then we have: \\
$(1)$ $\mu_i \in \mathcal{IS}^{(I)}(A)$ for $i=1,2,\cdots,10;$ \\
$(2)$ $\mu_3,\mu_5,\mu_6,\mu_9,\mu_{10} \in \mathcal{IS}^{(I)}_n(A);$ \\
$(3)$ $\mu_6,\mu_{10} \in \mathcal{IS}^{(II)}_n(A);$ \\ 
$(4)$ $\Ker(\mu_6)=\{1,a,b,c,d\}\in {\mathcal DS}_c(A)$ and $\Ker(\mu_{10})=A\in {\mathcal DS}_c(A)$. 
\end{ex}

\begin{ex} \label{st-psBCK-80} Consider the set $A=\{a,b,c,1\}$ and the operation 
$\rightarrow$ given by the following table:
\[
\hspace{10mm}
\begin{array}{c|cccc}
\rightarrow & a & b & c & 1 \\ \hline
a & 1 & c & 1 & 1 \\ 
b & c & 1 & 1 & 1 \\ 
c & c & c & 1 & 1 \\ 
1 & a & b & c & 1  
\end{array}
. 
\]
Then $(A,\rightarrow,1)$ is a commutative BCK-algebra (see \cite[Ex. 5.1.13]{DvPu}). \\
Consider the maps $\mu_i:A\longrightarrow A$, $i=1,2,3,4$, given in the table below:
\[
\begin{array}{c|cccc}
 x & a & b & c & 1 \\ \hline
\mu_1(x) & a & a & c & 1 \\
\mu_2(x) & a & b & c & 1 \\
\mu_3(x) & b & b & c & 1 \\
\mu_4(x) & 1 & 1 & 1 & 1 
\end{array}
.   
\]
Then, for $i=1,2,3,4$, we have: \\
$(1)$ $\mu_i \in \mathcal{IS}^{(I)}_n(A)=\mathcal{IS}^{(II)}_n(A),$ \\
$(2)$ $\Ker(\mu_i)$ is a commutative deductive system of $A$.  
\end{ex}

$\vspace*{5mm}$

\section{State-morphism pseudo BCK-algebras}

In this section we introduce and study the notion of  state-morphism operators on a pseudo BCK-algebra showing 
that any normal type II state operator on a linearly ordered pseudo BCK-algebra is a state-morphism. 
For the case of a linearly ordered commutative pseudo BCK-algebra, we prove that any type I or type II state operator 
is also a state-morphism operator. 

\begin{Def} \label{sm-bck-10} Let $(A, \rightarrow,\rightsquigarrow,1)$ be a pseudo BCK-algebra. 
A homomorphism $\mu:A\longrightarrow A$ is called a \emph{state-morphism operator} on $A$ if $\mu^2=\mu$, where $\mu^2=\mu\circ \mu$. The pair $(A, \mu)$ is called a \emph{state-morphism pseudo BCK-algebra}.
\end{Def}

Denote $\mathcal{SM}(A)$ the set of all state-morphism operators on a pseudo BCK-algebra $A$. 

\begin{exs} \label{sm-bck-20}
$(1)$ $1_A, Id_A\in \mathcal{SM}(A)$ for any pseudo BCK-algebra $A$. \\
$(2)$ If $A$ is the pseudo BCK-algebra from Example \ref{st-psBCK-70} then 
$\mathcal{SM}(A)=\{\mu_3, \mu_6, \mu_{10}\}$. \\
$(3)$ If $(A,\rightarrow,\rightsquigarrow,0,1)$ is a bounded commutative pseudo BCK-algebra and  
$\mu:A\longrightarrow A$ defined by $\mu(x)=x^{\rightarrow_0\rightsquigarrow_0}$ then $\mu\in \mathcal{SM}(A)$.
\end{exs}

\begin{ex} \label{sm-bck-20-10} Let $(A_1, \rightarrow_1, \rightsquigarrow_1, 1_1)$ and 
$(A_2, \rightarrow_2, \rightsquigarrow_2, 1_2)$ be two pseudo BCK-algebras and let $A$ be the pseudo BCK-algebra 
defined in Example \ref{st-psBCK-10-30}. Then the maps $\mu_1, \mu_2:A\longrightarrow A$ defined by 
$\mu_1((x,y))=(x,x)$ and $\mu_2((x,y))=(y,y)$, for all $(x,y)\in A$ are state-morphism operators on $A$.
\end{ex}

\begin{rem} \label{sm-bck-20-20} 
By Lemma \ref{ps-eq-40}$(6)$, every state-morphism pseudo BCK-algebra is a type I state pseudo BCK-algebra, 
but not every state-morphism pseudo BCK-algebra is a type II state pseudo BCK-algebra. 
For example $Id_A$ is a state-morphism operator and a type I state operator on the pseudo BCK-algebra $A$, 
but it is a type II state operator iff $A$ is a commutative pseudo BCK-algebra. 
\end{rem}

\begin{prop} \label{sm-bck-30} Let $(A, \rightarrow,\rightsquigarrow,1)$ be a linearly ordered pseudo BCK-algebra. 
Then: \\
$(1)$ Any normal type II state operator on $A$ is a state-morphism operator; \\
$(2)$ If $A$ is commutative then any state operator on $A$ is a state-morphism operator.
\end{prop} 
\begin{proof}
$(1)$ Let $\mu$ be a normal type II state operator on $A$. 
According to Proposition \ref{st-psBCK-40}, $\mu$ is also a type I state operator.
If $x\le y$ then by Proposition \ref{st-psBCK-20}$(1)$, $1=\mu(1)=\mu(x\rightarrow y)\le \mu(x)\rightarrow \mu(y)=1$. 
Hence $\mu(x\rightarrow y)=\mu(x)\rightarrow \mu(y)$ and similarly 
$\mu(x\rightsquigarrow y)=\mu(x)\rightsquigarrow \mu(y)$. \\
If $y\le x$ then by $(IS^{'}_2)$ we have $\mu(x\rightarrow y)=\mu(x)\rightarrow \mu(y)$ and  
$\mu(x\rightsquigarrow y)=\mu(x)\rightsquigarrow \mu(y)$. 
It follows that $\mu$ is a homomorphism on $A$. Since by Proposition \ref{st-psBCK-20}$(2)$, $\mu^2=\mu$, 
we conclude that $\mu$ is a state-morphism operator on $A$. \\
$(2)$ If $x\le y$, similarly as in $(1)$ we get $\mu(x\rightarrow y)=\mu(x)\rightarrow \mu(y)$ and  
$\mu(x\rightsquigarrow y)=\mu(x)\rightsquigarrow \mu(y)$. 
If $y\le x$ then 
$\mu(x\rightarrow y)=\mu((y\rightarrow x)\rightsquigarrow x)\rightarrow \mu(y)=\mu(x)\rightarrow \mu(y)$ and 
similarly $\mu(x\rightsquigarrow y)=\mu(x)\rightsquigarrow \mu(y)$. 
Thus $\mu$ is a state-morphism operator on $A$. 
\end{proof}

\begin{rem} \label{sm-bck-30-10} In general a type I state operator on a linearly ordered pseudo BCK-algebra $A$ 
is not necessarily a state-morphism operator on $A$. 
Indeed consider the set $A=\{0,a,b,1\}$, where $0<a<b<1$, and the operations $\rightarrow,\rightsquigarrow$ given by the following tables: 
\[
\hspace{10mm}
\begin{array}{c|cccc}
\rightarrow & 0 & a & b & 1 \\ \hline
0 & 1 & 1 & 1 & 1 \\ 
a & a & 1 & 1 & 1 \\ 
b & a & a & 1 & 1 \\ 
1 & 0 & b & b & 1 
\end{array}
\hspace{10mm} 
\begin{array}{c|ccccccc}
\rightsquigarrow & 0 & a & b & 1 \\ \hline
0 & 1 & 1 & 1 & 1 \\ 
a & a & 1 & 1 & 1 \\ 
b & a & a & 1 & 1 \\ 
1 & 0 & b & b & 1 
\end{array}
. 
\] 
Then $(A,\rightarrow,\rightsquigarrow,1)$ is a linearly ordered pseudo BCK-algebra (see \cite{Kuhr6}). 
Since $(a\rightarrow b)\rightsquigarrow b=b\neq 1=(b\rightarrow a)\rightsquigarrow a$, it follows that $A$ is 
not commutative. 
We can easily see that the map $\mu:A\longrightarrow A$, defined by $\mu(0)=\mu(a)=a, \mu(b)=\mu(1)=1$ is a type I 
state operator on $A$, but it is not a state-morphism operator. 
\end{rem}

For $\mu \in \mathcal{SM}(A)$, $\Ker(\mu)=\{x\in A \mid \mu(x)=1\}$ is called the \emph{kernel} of $\mu$.

\begin{lemma} \label{sm-bck-30-20} For any $\mu \in \mathcal{SM}(A)$, $\Ker(\mu)\in {\mathcal DS}_n(A)$. 
\end{lemma}
\begin{proof} The proof is straightforward. 
\end{proof}

\begin{prop} \label{sm-bck-30-30} Let $(A, \mu)$ be a commutative state-morphism pseudo BCK-algebra. 
Then the following hold: \\
$(1)$ $\Img(\mu)$ is a commutative pseudo BCK-algebra; \\
$(2)$ $\mu(x\vee y)=\mu(x)\vee \mu(y)$ and $\mu(\mu(x)\vee \mu(y))=\mu(x)\vee \mu(y)$, for all $x, y\in A$.  
\end{prop}
\begin{proof} Since $\mu$ is a type I state operator on $A$, by Remark \ref{sm-bck-20-20}, applying 
Proposition \ref{st-psBCK-20}$(5)$ it follows that $\Img(\mu)$ is a subalgebra of $A$.  
The rest of the proof is straightforward. 
\end{proof}

\begin{prop} \label{sm-bck-40} Let $(A, \mu)$ be a state-morphism pseudo BCK-algebra. The following hold: \\
$(1)$ $\Ker(\mu)=\{\mu(x)\rightarrow x \mid x\in A\}=\{x\rightarrow \mu(x) \mid x\in A\}$ \\
$\hspace*{1.9cm}$
               $=\{\mu(x)\rightsquigarrow x \mid x\in A\}=\{x\rightsquigarrow \mu(x) \mid x\in A\};$ \\
$(2)$ if $\Ker(\mu)=\{1\}$ then $\mu=Id_A$. 
\end{prop}
\begin{proof}
$(1)$ Since $\mu^2=\mu$ and $\mu$ is a homomorphism, we have 
$\mu(\mu(x)\rightarrow x)=\mu(x)\rightarrow \mu(x)=1$, so $\mu(x)\rightarrow x\in \Ker(\mu)$. 
If $x\in \Ker(\mu)$, we have $x=1\rightarrow x=\mu(x)\rightarrow x \in \{\mu(x)\rightarrow x \mid x\in A\}$. 
Hence $\Ker(\mu)=\{\mu(x)\rightarrow x \mid x\in A\}$. \\
Similarly $\Ker(\mu)=\{x\rightarrow \mu(x) \mid x\in A\}=\{\mu(x)\rightsquigarrow x \mid x\in A\}=
\{x\rightsquigarrow \mu(x) \mid x\in A\}$. \\
$(2)$ For any $x\in A$, $\mu(x)\rightarrow x, x\rightarrow \mu(x)\in \Ker(\mu)=\{1\}$, hence $\mu(x)=x$. 
\end{proof}

\begin{cor} \label{sm-bck-40-10} If $A$ is a simple pseudo BCK-algebra then $\mathcal{SM}(A)=\{1_A,Id_A\}$. 
\end{cor}

\begin{Def} \label{sm-bck-40-20} Let $(A, \mu)$ be a state-morphism pseudo BCK-algebra. A deductive system $D$ 
of $A$ is called a \emph{$\mu$-state deductive system} if $\mu(D)\subseteq D$. 
Denote by ${\mathcal SDS}_{\mu}(A)$ the set of all $\mu$-state deductive systems of $(A,\mu)$.
\end{Def} 

\begin{theo} \label{sm-bck-50} Let $(A, \mu)$ be a state-morphism pseudo BCK-algebra. The following hold: \\
$(1)$ $\mu$ is injective iff $\Ker(\mu)=\{1\};$ \\
$(2)$ if $D\in {\mathcal DS}(A)$ then $\mu^{-1}(D)\in {\mathcal DS}(A)$ and $\Ker(\mu)\subseteq \mu^{-1}(D);$ \\
$(3)$ if $D\in {\mathcal DS}_n(A)$ then $\mu^{-1}(D)\in {\mathcal DS}_n(A);$ \\
$(4)$ if $D\in {\mathcal DS}_c(A)$ then $\mu^{-1}(D)\in {\mathcal DS}_c(A);$ \\
$(5)$ if $\mu$ is surjective and $D\in {\mathcal SDS}_{\mu}(A)$ then 
$\mu(D), \mu(\mu(D))\in {\mathcal SDS}_{\mu}(A)$. 
\end{theo}
\begin{proof}
$(1)$ Suppose that $\mu$ is injective and let $x\in \Ker(\mu)$. It follows that $\mu(x)=1=\mu(1)$, hence $x=1$. 
Conversely, suppose that $\Ker(\mu)=\{1\}$ and consider $x, y\in A$ such that $\mu(x)=\mu(y)$. 
We get $\mu(x\rightarrow y)=\mu(x)\rightarrow \mu(y)=1$, so $x\rightarrow y\in \Ker(\mu)$. 
It follows that $x\rightarrow y=1$, that is $x\le y$. Similarly we have $y\le x$, so $x=y$. 
Thus $\mu$ is injective. \\
$(2)$ Consider $D\in {\mathcal DS}(A)$. Since $1\in D$, we have $\mu(x)=1\in D$ for all $x\in \Ker(\mu)$, 
hence $\Ker(\mu)\subseteq \mu^{-1}(D)$. 
Let $x, x\rightarrow y\in \mu^{-1}(D)$, that is $\mu(x), \mu(x\rightarrow y)\in D$.  
It follows that $\mu(x), \mu(x)\rightarrow \mu(y)\in D$, so $\mu(y)\in D$. 
Hence $y\in \mu^{-1}(D)$, that is $\mu^{-1}(D)\in {\mathcal DS}(A)$. \\
$(3)$ Suppose $D\in {\mathcal DS}_n(A)$. Then $x\rightarrow y\in \mu^{-1}(D)$ iff $\mu(x\rightarrow y) \in D$ iff 
$\mu(x)\rightarrow \mu(y) \in D$ iff $\mu(x)\rightsquigarrow \mu(y) \in D$ iff $\mu(x\rightsquigarrow y) \in D$ iff 
$x\rightsquigarrow y\in \mu^{-1}(D)$. Hence $\mu^{-1}(D)\in {\mathcal DS}_n(A)$. \\ 
$(4)$ Consider $D\in {\mathcal DS}_c(A)$ and let $x, y\in A$ such that $y\rightarrow x\in \mu^{-1}(D)$, that is 
$\mu(y\rightarrow x)\in D$, so $\mu(y)\rightarrow \mu(x)\in D$. Since $D$ is commutative, it follows that 
$((\mu(x)\rightarrow \mu(y))\rightsquigarrow \mu(y))\rightarrow \mu(x)\in D$, so 
$\mu(((x\rightarrow y)\rightsquigarrow y)\rightarrow x)\in D$. 
Hence $((x\rightarrow y)\rightsquigarrow y)\rightarrow x\in \mu^{-1}(D)$. 
Similarly $y\rightsquigarrow x\in \mu^{-1}(D)$ implies 
$((x\rightsquigarrow y)\rightarrow y)\rightsquigarrow x\in \mu^{-1}(D)$. 
Hence $\mu^{-1}(D)\in {\mathcal DS}_c(A)$. \\
$(5)$ Since $1\in D$, we have $1=\mu(1)\in \mu(D)$. 
Consider $x\in \mu(D)$, $y\in A$ such that $x\rightarrow y\in \mu(D)$. 
There exists $x_1\in D$ such that $\mu(x_1)=x$, and from the surjectivity of $\mu$, there exists $y_1\in A$ 
such that $\mu(y_1)=y$. 
We have $x\rightarrow y\in \mu(D)$ iff $\mu(x_1)\rightarrow \mu(y_1)\in \mu(D)$ iff 
$\mu(x_1\rightarrow y_1)\in \mu(D)\subseteq D$, so $x_1\rightarrow y_1\in \mu^{-1}(D)$. 
Since $\mu^{-1}(D)\in {\mathcal DS}(A)$, we have $y_1\in \mu^{-1}(D)$, that is $y=\mu(y_1)\in \mu(D)$. 
Hence $\mu(D)\in {\mathcal DS}(A)$, so $\mu(D)\in {\mathcal SDS}_{\mu}(A)$. \\
Let $x\in \mu(\mu(D))$, so there exist $x_1\in \mu(D)$ and $x_2\in D$ such that $x=\mu(x_1)$ and $x_1=\mu(x_2)$. 
It follows that $x=\mu(\mu(x_2))=\mu(x_2)=x_1\in \mu(D)$, hence $\mu(\mu(D))\subseteq \mu(D)$. 
Since $\mu(D)\in {\mathcal DS}(A)$, we get $\mu(\mu(D))\in {\mathcal SDS}_{\mu}(A)$.  
\end{proof}

\begin{cor} \label{sm-bck-50-10} If $D\in {\mathcal SDS}_{\mu}(A)$ with $\mu$ surjective then 
$\mu^n(D)\in {\mathcal SDS}_{\mu}(A)$, for all $n\in{\mathbb N}$, $n\ge 2$, where $\mu^n=\mu^{n-1}\circ \mu$. 
\end{cor}

\begin{prop} \label{sm-bck-60} Let $(A, \mu)$ be a state-morphism pseudo BCK-algebra and the map 
$\hat \mu:A/\Ker(\mu) \to A/\Ker(\mu)$, defined by $\hat \mu(x/\Ker(\mu))= \mu(x)/\Ker(\mu)$. 
Then $(A/\Ker(\mu), \hat \mu)$ is a state-morphism pseudo BCK-algebra.  
\end{prop} 
\begin{proof}
If $x/\Ker(\mu)=y/\Ker(\mu)$ then $x\rightarrow y, y\rightarrow x\in \Ker(\mu)$, that is 
$\mu(x\rightarrow y)=\mu(y\rightarrow x)=1$. 
It follows that $\mu(x)\rightarrow \mu(y)=\mu(y)\rightarrow \mu(x)=1$, so $\mu(x)=\mu(y)$.
Hence $\hat \mu$ is well defined. 
The proof of the fact that $\hat \mu$ is a state-morphism on $A/\Ker(\mu)$ is straightforward. 
\end{proof}

\begin{theo} \label{sm-bck-70} Let $(A, \mu)$ be a state-morphism pseudo BCK-algebra and $\pi=\pi_{\Ker(\mu)}$ 
be the canonical projection from $A$ to $A/\Ker(\mu)$. Then $\pi\circ \mu=\pi$. 
\end{theo}
\begin{proof}
By Proposition \ref{sm-bck-40}$(1)$ for all $x\in A$, $(x,\mu(x))\in \Theta_{\Ker(\mu)}$, thus we have $\mu(x)/\Ker(\mu)=x/\Ker(\mu)$, so $\pi(\mu(x))=\mu(x)/\Ker(\mu)=x/\Ker(\mu)=\pi(x)$. 
Hence $\pi\circ \mu=\pi$.
\end{proof}

$\vspace*{5mm}$

\section{Concluding remarks}

In this paper we generalized the equational bases given by M. Pa\l asi\'nski and B. Wo\'zniakowska in \cite{Pal1} 
and by W.H. Cornish in \cite{Cor1} for commutative BCK-algebras to the case of commutative pseudo BCK-algebras. 
We proved certain results regarding the commutative deductive systems of pseudo BCK-algebras and we 
gave a characterization of commutative pseudo BCK-algebras in terms of commutative deductive systems. 
We applied these results to investigate the state pseudo BCK-algebras and the state-morphism pseudo BCK-algebras. 
As another direction of research, one could investigate the \emph{$n$-fold commutative pseudo BCK-algebras} and 
the \emph{$n$-fold commutative deductive systems} of pseudo BCK-algebras. 
We recall that a BCK-algebra $(A,*,0)$ is called $n$-fold commutative if there exists a fixed natural number $n$ 
such that the identity $x* y=x*(y*(y*x^n))$, where $y*x^0=y$, $y*x^{n+1}=(y*x^n)*x$, holds for all 
$x, y\in A$ (see \cite{Xie1}). 
An ideal $I$ of a BCK-algebra $(A,*,0)$ is called $n$-fold commutative if there exists a fixed natural number $n$ 
such that $x*y\in I$ implies $x* y=x*(y*(y*x^n))\in I$ for all $x, y\in A$ (see \cite{Huang1}). 
State pseudo BCK-algebras and state-morphism pseudo BCK-algebras could be studied in the context of 
$n$-fold commutative pseudo BCK-algebras and $n$-fold commutative deductive systems. 
These results could be extended to more general structures, such as pseudo BE-algebras and pseudo BCI-algebras.




$\vspace*{5mm}$

\setlength{\parindent}{0pt}

\end{document}